\documentclass[11pt, a4paper]{article}
\usepackage[a4paper, margin=1in]{geometry} 

\usepackage[T1]{fontenc}
\usepackage{lmodern}

\usepackage[a4paper, margin=1in]{geometry}
\usepackage{setspace}
\usepackage{amssymb, amsfonts, amsthm, amsmath, mathtools}
\usepackage{nicefrac}
\usepackage{bm, bbm}
\usepackage{enumitem}
\usepackage{booktabs, multirow, multicol, makecell} % Tables.
\usepackage{comment}
\usepackage{hyperref}

\newtheoremstyle{proclaim}
  {3pt}        % Space above.
  {3pt}        % Space below.
  {\slshape}  % Body font.
  {}        % Indent amount.
  {\bfseries} % Theorem head font.
  {.}       % Punctuation after theorem head.
  { }    % Space after theorem head.
  {}        % Theorem head spec.

\theoremstyle{proclaim}

\newtheorem{lemma}{Lemma}
\newtheorem{proposition}{Proposition}
\newtheorem{theorem}{Theorem}
\newtheorem{corollary}{Corollary}

\theoremstyle{definition}

\newtheorem{example}{Example}

\theoremstyle{remark}
\newtheorem*{remark*}{Remark}

\usepackage{appendix}

\newcommand{\blank}{\kern.15em\cdot\kern.15em}% Blank function input.
\makeatletter
\newcommand*{\transpose}{{\mathpalette\@transpose{}}}
\newcommand*{\@transpose}[2]{\raisebox{\depth}{$\m@th#1\kern0.0556em\intercal$}}
\makeatother

\newcommand{\awDistance}{d\kern-.12em l} % Attouch-Wets distance.
\newcommand{\gtrasymp}{\mathrel{\raisebox{2.25pt}{$>$}\kern-0.75em\raisebox{-3.25pt}{$\smallfrown$}}}

\newcommand{\tspace}{\kern0.0556em} % Tiny space
\newcommand{\ttspace}{\tspace\tspace}
\newcommand{\ntspace}{\kern-0.0556em}

\newcommand{\reals}{\mathbb{R}}

\newcommand{\ball}{\mathbb{B}}
\newcommand{\nats}{\mathbb{N}}

\newcommand{\tsum}[2]{\mathop{{\textstyle{\sum_{#1}^{#2}}}}}
\newcommand{\tMeasureSum}{{\textstyle\sum}}

\newcommand{\tinf}[1]{\mathop{{\textstyle{\inf_{#1}}}}}
\newcommand{\tmin}[1]{\mathop{{\textstyle{\min_{#1}}}}}

\DeclareMathAlphabet{\mathsl}{T1}{cmr}{m}{sl}
\newcommand{\drv}{\mathsl{d}}

\newcommand{\defeq}{\coloneq}
\newcommand{\eqdef}{\eqcolon}
\DeclareMathOperator{\diam}{diam}

\DeclareMathOperator*{\minimize}{minimize}
\DeclareMathOperator*{\maximize}{maximize}

\newcommand{\subjectTo}{\text{subject to}}

\usepackage[stretch = 30, shrink = 20]{microtype} % Font expansion of up to 3 % on loose lines. (Default = 20, 20.)
\AtBeginDocument{\hyphenpenalty = 1000 } % (Default = 50.)
\newcount\shortPageLines \shortPageLines = 2 % Maximum number of lines left short. (0 turns this off.)
\newcount\shortPageBadness \shortPageBadness = 300 % Badness of stretch allowed before a line is left short.
\makeatletter
\newdimen\shortPage@line % Height of one line.
\newdimen\shortPage@skip % Space left at the bottom of the page.
\AtBeginDocument{\ifnum\shortPageLines>\z@
  \begingroup \normalsize \global\shortPage@line = \baselineskip \endgroup
  \topskip = 1\topskip plus \shortPage@line % Makes ending a page one line early cheap. Cancelled below.
  \let\shortPage@makecol\@makecol
  \let\@makecol\shortPage@wrap
\fi}
\def\shortPage@wrap{%
  \global\shortPage@skip\z@
  \ifx\@textbottom\relax \ifvbox\@kludgeins \else % Only under \flushbottom and without \enlargethispage.
    \shortPage@measure
    \count@\clubpenalties\@ne \ifnum\widowpenalties\@ne<\count@ \count@\widowpenalties\@ne \fi
    \ifnum\count@>\z@ \ifnum\outputpenalty<\@M \ifnum\outputpenalty<\count@ \else % Break at a club or widow line.
      \@latex@warning@no@line{A club or widow line was left at the break after page \thepage}%
    \fi\fi\fi
  \fi\fi
  \setbox\@cclv\vbox{\boxmaxdepth\@maxdepth
    \vskip\z@ plus -\gluestretch\topskip % Cancel the stretch of \topskip.
    \unvbox\@cclv
    \ifdim\shortPage@skip>\z@ \vskip\shortPage@skip \fi}% Space below the text and above footnotes.
  \shortPage@makecol}
\def\shortPage@measure{%
  \begingroup
    \vbadness\@M \vfuzz\maxdimen % Silence the trial boxes.
    \setbox\@tempboxa\vbox{\boxmaxdepth\@maxdepth % Text and footnotes as LaTeX stacks them.
      \vskip\z@ plus -\gluestretch\topskip \unvcopy\@cclv
      \ifvoid\footins \else \vskip\skip\footins \footnoterule \unvcopy\footins \fi}%
    \dimen@\z@ \ifx\@botlist\@empty \dimen@\dp\@tempboxa \fi
    \count@\z@ \shortPage@try
    \ifnum\count@>\z@
      \ifvoid\footins \else \dimen@\z@ \fi
      \global\shortPage@skip = \dimexpr\count@\shortPage@line-\dimen@\relax
    \fi
  \endgroup}
\def\shortPage@try{%
  \setbox\z@\vbox to\vsize{\unvcopy\@tempboxa \vskip-\dimen@ \vskip\count@\shortPage@line}%
  \let\next\relax
  \ifnum\badness>\@M \advance\count@\m@ne % Overfull, so one line too many.
  \else \ifnum\badness>\shortPageBadness \ifnum\count@<\shortPageLines
    \advance\count@\@ne \let\next\shortPage@try
  \fi\fi\fi
  \next}
\makeatother

\usepackage{csquotes} % Recommended for babel.
\usepackage[british]{babel} % British patterns.
\newcommand{\rxi}{\bm{\xi}}
\newcommand{\rzeta}{\bm{\zeta}}

\newcommand{\Expt}{\mathbb{E}} % Expectation.
\newcommand{\Prob}{\mathbb{P}} % Probability.
\newcommand{\Qrob}{\mathbb{Q}} % Alternate probability.
\newcommand{\pmProb}{\mathbbm{1}} % Point-mass probability.
\newcommand{\given}{\kern.1em\vert\kern.1em}
\newcommand{\bigGiven}{\kern-.15em\bigm\vert\kern-.15em}
\newcommand{\BigGiven}{\kern-.15em\Bigm\vert\kern-.15em}
\newcommand{\biggGiven}{\kern-.15em\biggm\vert\kern-.15em}
\newcommand{\BiggGiven}{\kern-.15em\Biggm\vert\kern-.15em}

\usepackage{xcolor}
\usepackage{graphicx}
\usepackage{tikz}
\usetikzlibrary{positioning}
\usepackage[labelformat=simple]{subcaption}
\usepackage{float} % Figure autopositioning ([H])

\definecolor{myblue}{RGB}{0, 119, 180}

\usepackage[normalem]{ulem}

\newcommand{\daySuffix}[1]{%
  \begingroup
    \edef\lastdigit{\number\numexpr #1 - 10*(#1/10)\relax}%
    \edef\lasttwo{\number\numexpr #1 - 100*(#1/100)\relax}%
    \ifnum\lasttwo=11 th%
    \else\ifnum\lasttwo=12 th%
    \else\ifnum\lasttwo=13 th%
    \else
      \ifcase\lastdigit
        th% 0
      \or st% 1
      \or nd% 2
      \or rd% 3
      \else th% 4–9
      \fi
    \fi\fi\fi
  \endgroup
}

\newcommand{\wordMonth}{\ifcase \month \or January \or February \or March \or April \or May \or June \or July \or August \or September \or October \or November \or December \fi}

\renewcommand{\epsilon}{\varepsilon}

\newcommand{\cA}{{\cal A}}

\newcommand{\cP}{{\cal P}}

\newcommand{\cX}{{\cal X}}

\usepackage[backend = biber,
            giveninits = false,
            maxbibnames = 10,
            doi = false, 
            isbn = false, 
            url = false, 
            eprint = false,
            style = ext-numeric, % `ext-' important to get 'titlecase:title' specification to work.
            ]{biblatex}
\DeclareNameAlias{sortname}{given-family} 
\renewbibmacro{in:}{} % No `in journal' (good for presentations).
\DeclareDelimFormat{finalnamedelim}{
    \ifnumgreater{\value{liststop}}{2}
    {\unskip,\addspace\&} % If true (3+ authors)
    {\&} % If false (2 authors).
}

\DeclareFieldFormat[article,inbook,incollection,inproceedings,patent,thesis,unpublished]{titlecase:title}{\MakeSentenceCase*{#1}}

\DeclareFieldFormat[article,inbook,incollection,inproceedings,patent,thesis,unpublished]{title}{``#1\addperiod''} % American.

\usepackage{biblatex}

\newcommand{\Neff}{N_{\mathrm{eff}}}

\let\epsilon\relax
\newcommand{\epsilon}{\varepsilon}

\DeclareMathOperator{\logRange}{\texttt{LogRange}}
\DeclareMathOperator{\linRange}{\texttt{LinRange}}

\newcommand{\weights}{\Delta}
\renewcommand{\pmProb}{\delta}

\newcommand{\costUnder}{c_{\mathrm{u}}}
\newcommand{\costOver}{c_{\mathrm{o}}}

\bibliography{bib}

\makeatletter
\def\@maketitle{%
  \newpage
  \begin{center}%
  \let \footnote \thanks
    {\LARGE \@title \par}%
    \vskip 1.5em%
    {\large
      \lineskip .5em%
      \begin{tabular}[t]{c}%
        \@author
      \end{tabular}\par}%
  \end{center}%
  }%\vskip .25em}
\makeatother
\title{\textbf{Don't Look Back in Anger:\\Wasserstein Distributionally Robust Optimization\\with Nonstationary Data}}%%%
\author{\textsl{Dominic S.\ T.\ Keehan} \\ {\small Imperial Business School} \\ {\small \texttt{d.keehan@imperial.ac.uk}}\\ \and \textsl{Edward J.\ Anderson} \\ {\small Imperial Business School}\\ {\small \texttt{e.anderson@imperial.ac.uk}} \and \textsl{Wolfram Wiesemann} \\ {\small Imperial Business School}\\ {\small \texttt{ww@imperial.ac.uk}}}%%%

\begin{document}

\maketitle
%\vspace*{-28pt}%%%

\noindent \textbf{Abstract:} We study data-driven decision problems where historical observations are generated by a time-evolving distribution whose change over time is bounded in Wasserstein distance. We~address this nonstationarity using a distributionally robust optimization model with an ambiguity set that is a Wasserstein ball centered at a \emph{weighted} empirical distribution, thereby allowing for the time decay of past data in a way which accounts for the drift of the data-generating distribution. Our~main technical contribution is a concentration bound for weighted empirical distributions that explicitly captures both the effective sample size (i.e., the equivalent number of equally weighted observations) and the distributional drift. Using our concentration bound, we select observation weights that optimally balance the effective sample size against the distributional drift of the data-generating process, taking into account the order of the Wasserstein ball. We~also consider simple weighting schemes, and  find that across a wide range of problem settings, a triangular weighting scheme is an effective choice that outperforms the classical weighting schemes of time windowing and simple exponential smoothing. Numerical experiments on a multi-item newsvendor problem illustrate the benefits of our approach.

%We study data-driven decision problems where historical observations are generated by a time-evolving distribution whose change over time is bounded in Wasserstein distance. We address this nonstationarity using a distributionally robust optimization model with an ambiguity set that is a Wasserstein ball centered at a \emph{weighted} empirical distribution, thereby allowing for the time decay of past data in a way which accounts for the drift of the data-generating distribution. Our main technical contribution is a concentration bound for weighted empirical distributions that explicitly captures both the effective sample size (i.e., the equivalent number of equally weighted observations) and the distributional drift. Using our concentration bound, we select observation weights that optimally balance the effective sample size against the distributional drift of the data-generating process, taking into account the order of the Wasserstein ball. We also consider simple weighting schemes, and  find that across a wide range of problem settings, a triangular weighting scheme is an effective choice that outperforms the classical weighting schemes of time windowing and simple exponential smoothing. Numerical experiments on a multi-item newsvendor problem illustrate the benefits of our approach.

\bigskip%%%

\medskip%%%
\noindent \textbf{Date:} 19-th November 2026.% \number\day-\daySuffix{\day} \wordMonth \number\year.%%%

\section{Introduction}
Many practical decision problems involve problem data that are uncertain at the point of decision. Examples include inventory control with uncertain demands, maintenance scheduling under uncertain failures, online advertising with uncertain click-through rates, and portfolio allocation under uncertain market conditions.

Classical stochastic programming models the uncertain quantities as random variables and optimizes a risk measure---such as the expectation or (conditional) value-at-risk---under a known data-generating distribution. In~practice, however, this distribution is rarely known. Data-driven approaches therefore replace the unknown distribution with an estimate formed from historical observations. The~most prominent example is the sample average approximation~(SAA), which optimizes against the empirical distribution that places equal weight on past samples. While computationally appealing and asymptotically optimal, SAA is prone to overfitting in small- to moderate-sample regimes---a phenomenon often described as the ``optimizer’s curse.'' Distributionally robust optimization (DRO) mitigates overfitting by optimizing against the worst distribution in an ambiguity set that contains all distributions deemed plausible given the data. Among different ambiguity models, Wasserstein balls around an empirical reference distribution have emerged as a common choice since they respect the geometry of the sample space, yield finite-sample performance guarantees, and often lead to tractable formulations.

The vast majority of research on data-driven optimization presumes that observations are independent and identically distributed over time. In~practice, this assumption is often untenable: consumer preferences shift, supply conditions change, and competitors, whose actions are often modeled as exogenous uncertainty, adapt over time. Under nonstationarity, pooling all past observations with equal weight---as in standard SAA and Wasserstein DRO---can be harmful: looking at a long historical sequence of observations can increase statistical precision but may bias estimates toward outdated regimes, whereas focusing only on recent data can reduce bias~but~may~inflate~variance.

This paper proposes a Wasserstein DRO framework tailored to nonstationary data. We~model the evolution of the data-generating process as a sequence of unknown distributions, and we seek decisions that perform well under the \emph{next-period distribution}, which generates the next, as yet unobserved, realization. 
We~assume that the Wasserstein distances between the historical distributions and the next-period distribution are bounded by a \emph{drift profile}, that is, a sequence of bounds depending on the age of the observations. To~derive robust decisions, we center our ambiguity sets at \emph{weighted} empirical distributions that downweight older samples. We~then consider the fundamental variance--drift trade-off that governs data-driven decision making under nonstationarity. More specifically, our contributions are as follows.
\begin{enumerate}
    \item[{(i)}] \textbf{Modeling nonstationarity via Wasserstein shifts.}
    We~propose a DRO model that permits changes to the data-generating distribution over time, with the Wasserstein distance from each past distribution to the next-period distribution bounded by a drift profile. Our~ambiguity sets are centered at weighted empirical distributions, enabling the principled time decay of historical information within~DRO.
    \item[{(ii)}] \textbf{Weighted concentration for Wasserstein distances.} We~extend the classical concentration bounds of \parencite{fournier2015rate} from equally weighted to weighted empirical distributions and obtain finite-sample guarantees for nonstationary data, with rates that explicitly reflect both the effective sample size and the drift profile.
    \item[{(iii)}]\textbf{Optimizing the variance--drift trade-off.} We~show how to choose weights that balance statistical variance against distributional drift in our concentration bound. We~characterize the shape of the optimal weights and also study simple intuitive weighting schemes. We~find that a triangular scheme performs well across a range of problem settings and improves on
    classical schemes such as time windowing and simple exponential smoothing.
\end{enumerate}
Taken together, these contributions yield a conceptually simple and computationally practical method for decision making with nonstationary data: center a Wasserstein ambiguity set at a weighted empirical distribution, then calibrate the ambiguity radius and the parameters of the drift profile based on decision performance, with the corresponding weights determined by our theoretical characterization.

Our work intersects with, and contributes to, two strands of the literature: DRO under nonstationary data-generating processes and DRO with heterogeneous data sources.

The literature on DRO with nonstationary data relaxes the ubiquitous identically distributed assumption by modeling temporal dependence or drift. For~example, \parencite{DouAnitescu2019ORL, HuChenWang2025OR, SutterVanParysKuhn2020ArXiv} study Wasserstein DRO models where data are generated by vector autoregressive processes with unknown parameters. Likewise, \parencite{LiSutterKuhn2021ICML, SutterVanParysKuhn2020ArXiv} analyze Wasserstein DRO models where data arise from unknown finite-state Markov chains. In~addition, \parencite{PunWangYan2023MSOM} models nonstationarity via a regime-switching Wasserstein ambiguity set, and \parencite{DuchiGlynnNamkoong2021MOR} investigates $\phi$-divergence DRO where data are generated by certain classes of Markov chains. A~common feature of this literature is the imposition of parametric or Markovian structure on the temporal dynamics. Instead, we only assume that the Wasserstein distance between each past data-generating distribution and the next-period distribution is bounded, without committing to a specific parametric time-series model.

In contrast, the literature on DRO with heterogeneous data sources leverages samples from multiple distributions (sources) that are close to the data-generating (target) distribution in Wasserstein distance to construct an ambiguity set for the target. Focusing on two sources and least squares estimation, \parencite{TaskesenEtAl2021ICML} proposes either interpolating between the ambiguity sets centered at each source or intersecting them. For~logistic regression with two sources that are subjected to adversarial attacks, \parencite{SelviEtAl2025UAI} intersects the two ambiguity sets. In~the first work that studies more than two sources, \parencite{RychenerEtAl2024HeteroDRO} intersects Wasserstein balls centered at all sources and optimizes against the worst distribution in the intersection. The~paper shows that the problem is NP-hard in general but can be solved in polynomial time when either the number of sources or the dimension of the uncertain parameters is fixed. Finally, \parencite{GuoJiangShen2025MRDRO} studies a dynamic setting where source reliability varies over time and proposes mechanisms to update trust in each source. The~literature on DRO with heterogeneous data sources typically assumes that many  samples are available per source. On~the other hand, viewing each time step as a separate source, our setting provides exactly \emph{one} observation per source. We~will see that this renders the existing heterogeneous-source approaches unsuited. Instead, our approach aggregates over time via weighted empirical distributions while controlling for the Wasserstein drift.

Unlike the bandit and online optimization literature, which studies sequential decision making and regret minimization (see, e.g., \parencite{besbes2019optimal, bubeck-bandits, hazan-online, jiang2025online}), our goal is to construct a single data-driven ambiguity set with finite-sample coverage guarantees for the next-period distribution. Despite this difference in goals, the exponents that determine how our confidence radius and associated support length scale with the drift match the optimal regret and restarting exponents of \parencite{besbes-et-al-2015-nonstationary}, at~least for a wide range of drift values. We~discuss this in Section~\ref{sec:optimise-weights}.

The remainder of the paper is organized as follows. Section~\ref{sec:problem-setting} introduces our framework. Section~\ref{section:concentration-of-measure} develops new concentration bounds for weighted empirical distributions. Section~\ref{sec:optimise-weights} characterizes weight choices that optimally balance variance and drift. Section~\ref{sec:num-results} reports numerical results, and Section~\ref{sec:conclusions} discusses possible extensions. Supplementary material, source code, and results are available in the repository accompanying this work.%
\footnote{Link: \url{https://github.com/dominickeehan/dont-look-back-in-anger}.}%%%
%\footnote{Link: (see attached anonymized code).}%%%

\textbf{Notation.} $\lVert \blank \rVert$ denotes the standard Euclidean norm. $\mathfrak{P}(\Xi)$ denotes the set of all Borel probability measures on a set~$\Xi$. Among these measures, the point-mass distribution~$\pmProb_{\xi}$ assigns probability~$1$ to the outcome~$\xi \in \Xi$. For~a distribution~$\Prob\in\mathfrak{P} (\Xi)$, we write $\rxi\sim\Prob$ to express that a random vector~$\rxi$ is distributed according to $\Prob$. Boldface symbols denote random vectors, and plain symbols denote their outcomes. $\weights_N$ denotes the probability~$N$-simplex; that is, the set of non-negative $N$-tuples summing to $1$. Also, we write $[T]\defeq\{1,\dots,T\}$ and $x_+\defeq\max\{0,x\}$.

\section{Problem Setting}\label{sec:problem-setting}
We study data-driven decision problems where a decision maker observes 
$T$ historical samples ${\rxi_1 \sim \Prob_1}, \ldots, \rxi_T \sim \Prob_T$ of an uncertain parameter~$\xi \in\Xi$ governed by
unknown time-varying data-generating distributions $\Prob_1, \ldots, \Prob_T \in \mathfrak{P}(\Xi)$, each with supports contained in  the (known) closed set~$\Xi\subseteq\reals^m$. At~time~$T + 1$, the decision maker selects a measurable, extended real--valued loss function~$\ell:\Xi\to\reals\cup \{\infty\}$ from an admissible class~$\mathcal{L}$ to solve
\begin{equation}\label{eq:the-mother-of-all-problems}
    \minimize_{\ell \in \mathcal{L}} ~~ \sup_{\Qrob \in \mathcal{P}} \Expt_\Qrob\bigl[\ell(\rxi)\bigr],
\end{equation}
where $\mathcal{P} \subseteq \mathfrak{P}(\Xi)$ is the ambiguity set for the (unknown) next-period distribution~$\Prob_{T+1}$. In~the special case where losses are generated by a (known) cost function $f:\cX\times\Xi\to\reals\cup\{\infty\}$ depending on both a decision variable~$x\in \cX$ and an uncertain variable~$\xi\in\Xi$, the admissible class is $\mathcal{L}=\{\xi\mapsto f(x,\xi):x\in \cX\}$.
The~loss function in problem~\eqref{eq:the-mother-of-all-problems} may represent
cost or an increasing convex disutility of that cost, allowing the formulation to accommodate both risk-neutral and risk-averse preferences. In~this formulation there is a single decision problem, but the setting is dynamic in the sense that probability distributions change over time. In~many such applications decisions are required in multiple time periods, and our approach then follows the usual DRO modeling choice where a sequence of one-shot problems is solved as new data arrive. Each of these problems considers only the next-period loss and does not anticipate how future learning may affect subsequent decisions; accounting for such effects would require a dynamic formulation with additional assumptions.

When $\Prob_1 = \cdots = \Prob_{T+1}$ and $\rxi_1, \ldots, \rxi_{T}$ are drawn independently (the \emph{stationary} setting), a common choice for the ambiguity set in \eqref{eq:the-mother-of-all-problems} is $\cP = \ball_p \bigl(\tfrac{1}{T}\tMeasureSum_{t=1}^{T} \pmProb_{\rxi_t};\epsilon\bigr)$; the $p$-Wasserstein ball of radius~$\epsilon$ centered around the empirical distribution~$\tfrac{1}{T}\tMeasureSum_{t=1}^{T} \pmProb_{\rxi_t}$. Here
\begin{equation*}
    \ball_p(\Prob;\epsilon) \defeq \Bigl\{\Qrob\in\mathfrak{P}(\Xi):
    W_p(\Prob,\Qrob) \leq \epsilon \Bigr\},
    \quad
    W_p(\Prob,\Qrob) \defeq
    \Bigl(\tinf{\gamma \in \Gamma(\Prob,\Qrob)}
    \Expt_\gamma\bigl[\|\rxi-\rzeta\|^p\bigr]\Bigr)^{1/p},
\end{equation*}
and $\Gamma(\Prob,\Qrob)$ is the set of couplings of $\Prob$ and $\Qrob$.
This choice is popular because it yields
tractable reformulations of problem~\eqref{eq:the-mother-of-all-problems} for broad classes of loss functions and, under mild conditions, satisfies a finite-sample concentration guarantee of the form
\begin{equation*}
\Pr\biggl[\Prob_{T+1} \notin \ball_p\Bigl(\tfrac{1}{T}\tMeasureSum_{t=1}^{T} \pmProb_{\rxi_t};\epsilon\Bigr)\biggr]
    \leq c_1 \cdot\exp\Bigl(-c_2 T\epsilon^{\max\{m,2p\}}\Bigr),
\end{equation*}
for positive constants~$c_1$ and $c_2$ that depend only on $m$, $\Xi$, and $p$ (with a logarithmic correction in the case where $m=2p$). The~concentration bound provides a theoretical justification for this class of ambiguity sets. Since its constants are usually
unknown and the bound can be conservative, in practice one typically calibrates
the radius~$\epsilon$ via cross-validation.

We~depart from the stationary setting by allowing $\Prob_t$ to vary over time, but we assume that the next-period distribution~$\Prob_{T+1}$ does not drift too far from the historical distributions. In~other words, we stipulate that $W_p(\Prob_t,\Prob_{T+1}) \le \theta_{T-t+1}$ for all $t\in[T]$ and some \emph{drift profile}~$\theta\defeq(\theta_1,\ldots,\theta_T)\geq 0$.
To~focus on distributional shift rather than temporal dependence, we continue to assume that, given the distributions~$\Prob_1,\ldots,\Prob_{T}$, the samples~$\rxi_1, \ldots, \rxi_{T}$ are independent. Throughout the paper, we assume that the drift profile~$\theta_1,\ldots,\theta_T$ is known. This assumption allows our concentration bounds to explicitly capture how the profile influences the resulting probability bounds, the optimal weight choices, and the performance of simple weighting schemes. At~the same time, this assumption is not restrictive in practice: our results culminate in a family of optimal weighting schemes parametrized by the drift profile, and in applications one can calibrate
the parameters of a drift profile based on downstream decision~performance.

\begin{subequations}\label{eq:drift-profiles}
In our analysis, we focus in particular on \emph{power-law} drift profiles of the form
\begin{equation}\label{eq:power-law-drift-profile}
    \theta_{T-t+1} = (T-t+1)^\eta\rho, \qquad t\in[T],
\end{equation}
where $\rho\geq 0$ and $\eta\in[0,1]$. Allowing $\eta<1$ accommodates drift that grows sublinearly with the look-back horizon. For~example, if the mean vector of $\Prob_t$ follows a standard random walk, then its expected displacement over $T$ periods is of order~$\sqrt{T}$ rather than~$T$. In~such stochastic models, drift bounds may only hold with high probability, in which case our guarantees do not apply directly. Nevertheless, the optimal weights derived under such profiles still indicate which weights are likely to be effective in practice. The~special case~$\eta=1$ yields \emph{linear} drift profiles,
\begin{equation}\label{eq:linear-drift-profile}
    \theta_{T-t+1} = (T-t+1)\rho, \qquad t\in[T],
\end{equation}
which are implied by the natural inter-period drift bound $W_p(\Prob_t,\Prob_{t+1})\leq\rho$, $t\in[T]$. As~a simple example, this condition is satisfied if each distribution~$\Prob_{t+1}$ is identical to $\Prob_t$ up to shifts in the mean vector bounded by~$\rho$ in Euclidean distance.
\end{subequations}

In the presence of drifting distributions, the Wasserstein ambiguity set $\ball_p \bigl(\tfrac{1}{T}\tMeasureSum_{t=1}^{T} \pmProb_{\rxi_t};\epsilon\bigr)$ centered around the empirical distribution~$\tfrac{1}{T}\tMeasureSum_{t=1}^{T} \pmProb_{\rxi_t}$ overweights earlier observations that are less representative of the next-period distribution~$\Prob_{T+1}$. An~alternative suggested by the literature on DRO with heterogeneous sources is to intersect Wasserstein balls, with one ball centered at each single-sample ``data source'' estimate~$\pmProb_{\rxi_t}$. Thus, we set~$\mathcal{P} = \bigcap_{t=1}^{T} \ball_p(\pmProb_{\rxi_t};\epsilon_t)$ for some radii~$\epsilon_1,\ldots,\epsilon_T$ chosen to reflect that $\Prob_{T+1}$ is at a different distance from different distributions. We~now show that this approach exhibits undesirable behavior with many samples and in high dimensions.

\begin{example}[Undesirable behavior of the intersection approach]
\label{example:behaviour-of-the-intersection-approach}
For $\beta \in [0, 1)$, let~$\epsilon(\beta)$ denote the $(1-\beta)$-confidence $p$-Wasserstein radius (in the stationary setting) around a single-sample empirical distribution. If~the confidence radius is set from a concentration bound, then $\epsilon(\blank)$ is nonincreasing and strictly positive (see, e.g., \parencite[Theorem~2]{fournier2015rate} and~\parencite{fournier-2023-explicit-constants}). In~our time-varying setting with drift profile~$\theta_1,\ldots,\theta_T$, choose~$\beta_1,\ldots,\beta_T\in[0,1)$ such that $\sum_{t=1}^T \beta_t = \beta \in [0, 1)$ and set
    $\epsilon_t \defeq \epsilon(\beta_t) + \theta_{T-t+1}$ for all $t\in[T]$,
which achieves the guarantee \parencite[Proposition~4]{RychenerEtAl2024HeteroDRO}
    \begin{equation}\label{eq:intersection-guarantee}
        \Pr\Biggl[ \Prob_{T+1} 
        \notin \bigcap_{t=1}^{T} \ball_{p}(\pmProb_{\rxi_t};\epsilon_t) \Biggr]
        \leq \sum_{t=1}^T \beta_t = \beta.
    \end{equation}
We claim that this guarantee is unsatisfactory, as the following two special cases without drift show. Firstly, the intersection ambiguity set is asymptotically inconsistent. Indeed, consider the case where the data-generating process is stationary with $\theta_1=\cdots=\theta_T=0$ and $\Prob_1=\cdots=\Prob_{T+1}=\pmProb_\xi$ for some $\xi\in\Xi$. Then $\rxi_t=\xi$ almost surely for every $t\in[T]$, and
\[
    \bigcap_{t=1}^T
    \ball_p(\pmProb_{\rxi_t};\epsilon_t)
    =
    \ball_p\Bigl(
        \pmProb_\xi;
        \tmin{t\in[T]}\epsilon(\beta_t)
    \Bigr)
    \supseteq
    \ball_p\bigl(\pmProb_\xi;\epsilon(\beta)\bigr)\neq\{\pmProb_\xi\} .
\]
The inclusion follows because $\beta_t\leq\beta$ and $\epsilon(\blank)>0$ is nonincreasing. Hence, with zero drift, the intersection ambiguity set cannot collapse to $\pmProb_\xi$ as $T\to\infty$. This is because each constituent ball is centered at a single-sample empirical distribution, so the confidence radii do not decrease.

Secondly, due to the thin-shell phenomenon, this effect becomes more pronounced in high dimensions. Consider the stationary setting with $\Xi=\bigl\{\xi\in\reals^m:\|\xi\|\leq1\bigr\}$, $p=2$, and $\Prob_1=\cdots=\Prob_{T+1}=\text{Uniform}(\Xi)$, so that $\Expt_{\Prob_{T+1}}[\rxi]=0$ and $\Expt_{\Prob_{T+1}}\bigl[\|\rxi\|^2\bigr]=m/(m+2)$. We~have
\[
    W_2^2(\pmProb_{\xi},\Prob_{T+1})
    =
    \frac{m}{m+2}+\|\xi\|^2
    >
    \|\xi\|^2
    =
    W_2^2(\pmProb_{\xi},\pmProb_0)
    \quad\text{for all }\xi\in\Xi,
\]
so every ball centered at a single-sample empirical distribution that contains $\Prob_{T+1}$ must also contain~$\pmProb_0$. Hence, for the intersection to contain $\Prob_{T+1}$, it must also contain~$\pmProb_0$, for which $W_2(\pmProb_0,\Prob_{T+1}) = \sqrt{m/(m+2)} \to 1$ as $m \to \infty$. That is, in order for the intersection ambiguity set to contain the target distribution in high dimensions, it must also contain a distribution whose distance from the target is approximately as large as the radius of the entire domain.
\hfill\qedsymbol
\end{example}

%\newpage
To illustrate how the ambiguity sets of the intersection approach differ from those of the weighted approach developed in the following sections, consider a one-dimensional example with four observations $\xi_1=1$, $\xi_2=-1$, $\xi_3=2$, $\xi_4=3$, Wasserstein order~$p=2$, and a linear drift profile of the form~\eqref{eq:linear-drift-profile} with $\rho=1/3$. For~the weighted approach we set $\varepsilon=2.75$ and use weights $w_1=0$, $w_2=0.2$, $w_3=0.35$, $w_4=0.45$, as per Theorem~\ref{theorem:optimal-weights-p>=1} below. For~the intersection approach we use the radii construction of Example~\ref{example:behaviour-of-the-intersection-approach}, but scale $\varepsilon$ so that the sets are of comparable size. Figure~\ref{figure:comparing-ambiguity-sets} shows the distributions that lie in the ambiguity sets of each approach. (The~figure reports the means and standard deviations of all such distributions.) 
In~this example the intersection-based set is visibly more asymmetric.
The~weight-based set has a symmetric shape, and the Gelbrich bound implies that the mean--standard deviation pairs attainable within the ambiguity set form a truncated circular disk, which is symmetric about the vertical line through the weighted average of the observations.

\begin{figure}[tbp]
    \centering
    \includegraphics[width=\dimexpr 26.03pt + 240pt + 26.03pt\relax]{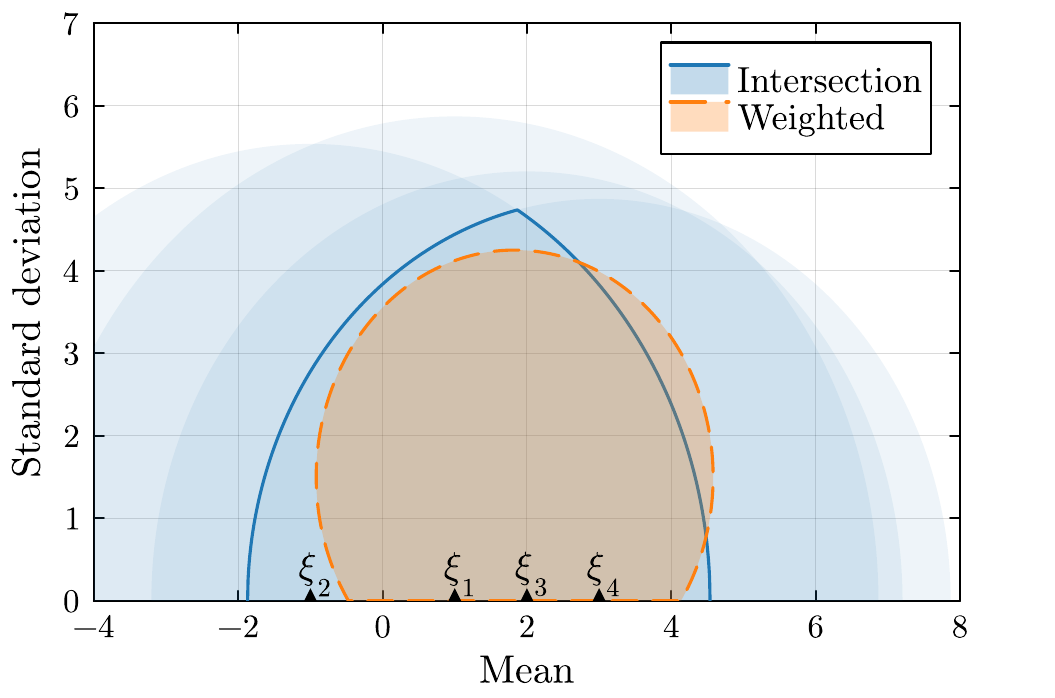}
    \caption{\textbf{Ambiguity Sets.} Means and standard deviations of distributions under the weighted and intersection approaches for a one-dimensional example with $p=2$.}
    \label{figure:comparing-ambiguity-sets}
\end{figure}

The undesirable behavior of the intersection approach established in Example~\ref{example:behaviour-of-the-intersection-approach} motivates our use of ambiguity sets centered at weighted empirical distributions~$\tMeasureSum_{t=1}^T w_t\pmProb_{\rxi_t}$, rather than intersections of balls around individual observations. The~next section develops finite-sample concentration bounds for such weighted empirical distributions under bounded drift, and Section~\ref{sec:optimise-weights} leverages these results to optimize the weights.

\section{Concentration of Measure for Weighted Empiricals}\label{section:concentration-of-measure}
The by-now classical results \parencite[Theorems~1 and~2]{fournier2015rate} show that when $\Xi\subseteq\reals^m$ is bounded\footnote{In fact, the results \parencite[Theorems~1 and~2]{fournier2015rate} apply with $\Xi$ unbounded, subject to certain moment conditions on the distributions; see the discussion after Theorem~\ref{theorem:time-heterogeneous-concentration-inequality} for a comparison of the results.} and $\rxi_1, \ldots, \rxi_N \sim \Prob \in\mathfrak{P}(\Xi)$ are independent and identically distributed random samples, then 
there exist constants $c_0, c_1, c_2 > 0$  and $q\in(0,1/2]$ (depending only on $m$, $\mathrm{diam}(\Xi)$, and $p$) such that
\begin{equation}\label{eq:FG_Thm1}
\Expt \biggl[W_p^p\Bigl(\tfrac{1}{N} \tMeasureSum_{i=1}^N \pmProb_{\rxi_i},\Prob \Bigr)\biggr] \leq c_0 N^{-q} \quad  \text{for all}~N\in\mathbb{N},
\end{equation}
and
\begin{equation}\label{eq:FG}
    \Pr\biggl[ W_p\Bigl(\tfrac{1}{N} \tMeasureSum_{i=1}^N \pmProb_{\rxi_i},\Prob\Bigr) \geq \epsilon\biggr]
    \leq c_1\cdot \exp \Bigl(-c_2 N\epsilon^{p/q}\Bigr)
    \quad\text{for all}~N\in\mathbb{N}, \; \epsilon \in \reals_+.
\end{equation}
Here $q \defeq \min\{p/m,1/2\}$ when $p \neq m/2$. The~boundary case~$p = m/2$ entails a logarithmic correction at $q=1/2$, which can be absorbed by taking $q$ slightly below $1/2$, at the cost of a weaker bound. We~therefore state our results for any $q\in(0,1/2)$ such that~\eqref{eq:FG_Thm1} holds.

The goal of this section is to derive an analogue of~\eqref{eq:FG} when the samples $\rxi_1 \sim \Prob_1, \ldots, \rxi_T \sim \Prob_T$ are independent but not identically distributed and when the equally weighted empirical distribution~$\tfrac{1}{T} \tMeasureSum_{t=1}^T \pmProb_{\rxi_t}$ is replaced by a weighted one, $\tMeasureSum_{t=1}^T w_t \pmProb_{\rxi_t}$ for $w \defeq (w_1, \ldots, w_T) \in \weights_T$. Our~bound will explicitly reflect the drift profile, and it will represent the weighting in the empirical distribution through the \emph{effective sample size}
\begin{equation*}
    \Neff(w_1,\ldots,w_N) \defeq \frac{1}{\sum_{i=1}^{N} w_i^2}.
\end{equation*}
This approximates 
the number of equally weighted samples that would yield the same variance as the weighted sample~\parencite{elvira2022rethinking, Kong1992}. It~ranges from $1$ to $N$, with a minimal value of $\Neff(w)=1$ when all of the weight is placed on a single observation and a maximal value of $\Neff(w)=N$ for uniform weights~$w_i = 1/N$. Throughout this section, our indices and weights match the setting: we~use $(i,N,w_i,\weights_N)$ for stationary results and $(t,T,w_t,\weights_T)$ for nonstationary~results.

To derive our results, we begin from the expectation bound of \textcite{fournier2015rate}, which is proved for equally weighted samples drawn independently from the same distribution. We~extend this to arbitrarily weighted empirical distributions in Proposition~\ref{proposition:weighted-empirical-Wasserstein-expectation-bound}. We~then turn to the nonstationary case where the samples are drawn from differing distributions. Lemma~\ref{lemma:shifted-samples-coupled-inequality} shows how the expected Wasserstein distance between the weighted empirical distribution formed from these nonstationary samples and the next-period distribution~$\Prob_{T+1}$ relates to the expected Wasserstein distance between a counterpart weighted empirical distribution formed from stationary samples and $\Prob_{T+1}$. Finally, Theorem~\ref{theorem:time-heterogeneous-concentration-inequality} combines Proposition~\ref{proposition:weighted-empirical-Wasserstein-expectation-bound} with Lemma~\ref{lemma:shifted-samples-coupled-inequality} and upgrades the resulting expectation bound to a high-probability concentration bound using a McDiarmid argument to deliver our main concentration bound for nonstationary data.

\begin{proposition}\label{proposition:weighted-empirical-Wasserstein-expectation-bound}
For a distribution~$\Prob\in\mathfrak{P}(\Xi)$ with bounded moments, let $q\in(0,1/2)$ achieve~\eqref{eq:FG_Thm1}. Then for independent and identically distributed random samples $\rxi_1,\ldots,\rxi_N\sim\Prob$, there exists a constant~$c_0>0$ (depending only on $m$, $\Prob$'s moment bounds, $p$, and $q$) such~that
\[
\Expt\biggl[W_p^p\Bigl(\tMeasureSum_{i=1}^N w_i\pmProb_{\rxi_i},\Prob\Bigr)\biggr]
\leq c_0\Neff(w)^{-q}
\quad\text{for all}~N\in\mathbb{N},\;w\in\weights_N.
\]
\end{proposition}

\begin{proof}
Our proof proceeds in three steps. Step~\hyperref[step:1]{1} decomposes the Wasserstein distance from the statement of the proposition into a sum of Wasserstein distances
from unweighted empirical distributions. Step~\hyperref[step:2]{2} applies the bound \eqref{eq:FG_Thm1} for expected Wasserstein distances to these unweighted empirical distributions. Finally, in Step~\hyperref[step:3]{3} we split the resulting sum of bounds into two parts, and show that each part is dominated by the expectation bound from the statement~of~the~proposition.

\emph{Step~1\label{step:1}.} We~reindex the weights so that $w_1 \ge \cdots \ge w_N$. This entails no loss of generality since reindexing does not change $\Neff(w)$. The~weighted empirical distribution satisfies
\begin{align*}
\sum_{i=1}^{N}w_{i}\pmProb_{\xi_{i}} &= w_N\sum_{i=1}^{N}\pmProb_{\xi_{i}} + (w_{N-1}-w_N)\ntspace\ntspace\sum_{i=1}^{N-1}\pmProb_{\xi_{i}} + \cdots + (w_1-w_2)\pmProb_{\xi_{1}} \\
&= Nw_N\frac{1}{N}\sum_{i=1}^{N}\pmProb_{\xi_{i}} + (N-1)(w_{N-1}-w_N)\frac{1}{N-1}\ntspace\ntspace\sum_{i=1}^{N-1}\pmProb_{\xi_{i}} + \cdots + (w_1-w_{2})\pmProb_{\xi_{1}},
\end{align*}
where the last expression constitutes a convex combination of the distributions $\tfrac{1}{N}\tMeasureSum_{i=1}^{N}\pmProb_{\xi_{i}}$, $\tfrac{1}{N-1}\tMeasureSum_{i=1}^{N-1}\pmProb_{\xi_{i}}$, $\ldots\,$, $\pmProb_{\xi_{1}}$ with coefficients $Nw_N$, $(N-1)(w_{N-1}-w_N)$, $\ldots\,$, $(w_{1}-w_2)$. Thus, by the convexity of $\mu \mapsto W_p^p(\mu,\Prob)$ (see, e.g., \parencite[Theorem~4.8]{Optimal-Transport:Villani}), we obtain
\begin{multline*}
W_p^p\Bigl(\tMeasureSum_{i=1}^{N}w_{i}\pmProb_{\xi_{i}},\Prob\Bigr)
\leq{} Nw_NW_p^p\Bigl(\tfrac{1}{N}\tMeasureSum_{i=1}^{N}\pmProb_{\xi_{i}},\Prob\Bigr) + (N-1)(w_{N-1}-w_N)W_p^p\Bigl(\tfrac{1}{N-1}\tMeasureSum_{i=1}^{N-1}\pmProb_{\xi_{i}},\Prob\Bigr)  \\
 + \cdots + (w_1-w_{2})W_p^p\Bigl(\pmProb_{\xi_{1}},\Prob\Bigr).
\end{multline*}

\emph{Step~2\label{step:2}.} We~use the expectation bound \eqref{eq:FG_Thm1} of \parencite[Theorem~1]{fournier2015rate} for each Wasserstein distance in the previous expression. Since $q\in(0,1/2)$ achieves \eqref{eq:FG_Thm1}, there is a constant~$c > 0$ (depending only on $m$, $\Prob$'s moment bounds, $p$, and $q$) such that
\begin{align*}
& \Expt \biggl[W_p^p\Bigl(\tMeasureSum_{i=1}^{N}w_{i}\pmProb_{\rxi_{i}},\Prob\Bigr)\biggr] \\
\leq{} & Nw_N c N^{-q} + (N-1)(w_{N-1}-w_N)c(N-1)^{-q} + \cdots + (w_{1}-w_{2})c \\
={} & N^{1-q}w_N c + (N-1)^{1-q}(w_{N-1}-w_N)c + \cdots + (w_{1}-w_{2})c\\
={} & \Bigl(N^{1-q}-(N-1)^{1-q}\Bigr)w_N c + \Bigl((N-1)^{1-q}-(N-2)^{1-q}\Bigr)w_{N-1}c + \cdots + w_1c.
\end{align*}
Thus, we obtain that
\begin{equation}\label{equation:expected-Wasserstein-distance-decomposition-sqrt-sum}
\Expt \biggl[W_p^p\Bigl(\tMeasureSum_{i=1}^{N}w_{i}\pmProb_{\rxi_{i}},\Prob\Bigr)\biggr] \leq c\sum_{i=1}^{N} \Bigl(i^{1-q}-(i-1)^{1-q}\Bigr)w_i \leq c\sum_{i=1}^{N} i^{-q}w_i,
\end{equation}
where the final inequality follows since $i^{1-q}-(i-1)^{1-q} \leq i^{-q}$ for all $i\in\nats$ when $q>0$.

\emph{Step~3\label{step:3}.} In~\eqref{equation:expected-Wasserstein-distance-decomposition-sqrt-sum} we decompose the final sum~$\sum_{i=1}^N i^{-q}w_i$ into $H + R$ with
\begin{equation*}
    H = \sum_{i=1}^{\lfloor\Neff(w)\rfloor} i^{-q}w_i  \quad \text{and} \quad R = \sum_{i=\lfloor \Neff(w) \rfloor+1}^{N} i^{-q}w_i.
\end{equation*}

First consider the ``head'' term~$H$. By~Cauchy--Schwarz and the definition of $\Neff(w)$,
\begin{equation*}
H \le \Bigl({\sum}_{i=1}^{\lfloor\Neff(w)\rfloor} i^{-2q}\Bigr)^{1/2}
       \Bigl({\sum}_{i=1}^{\lfloor\Neff(w)\rfloor} w_i^2\Bigr)^{1/2}
   \le \Bigl({\sum}_{i=1}^{\lfloor\Neff(w)\rfloor} i^{-2q}\Bigr)^{1/2}\Neff(w)^{-1/2}.
\end{equation*}
Since $x\mapsto x^{-2q}$ is decreasing, we have
\begin{equation*}
\sum_{i=1}^{\lfloor\Neff(w)\rfloor} i^{-2q}
\leq 1+\int_{1}^{\lfloor\Neff(w)\rfloor} x^{-2q}\,\drv x \leq \frac{{\lfloor\Neff(w)\rfloor}^{1-2q}}{1-2q},
\end{equation*}
and thus
\begin{align*}
H &\leq \frac{1}{\sqrt{1-2q}}{\lfloor\Neff(w)\rfloor}^{1/2-q}\Neff(w)^{-1/2} \leq \frac{1}{\sqrt{1-2q}}\Neff(w)^{-q},
\end{align*}
where the final inequality uses that $\lfloor\Neff(w)\rfloor \leq \Neff(w)$ and $1/2-q > 0$.

%\newpage%%%
Next we consider the ``remainder'' term~$R$. Since $i\mapsto i^{-q}$ is decreasing and $\sum_{i=1}^N w_i = 1$,
we~have
\begin{equation*}
R \leq \sum_{i=\lfloor\Neff(w)\rfloor+1}^N  \Bigl(\lfloor\Neff(w)\rfloor+1\Bigr)^{-q} w_i \leq \Bigl(\lfloor\Neff(w)\rfloor+1\Bigr)^{-q} \leq \Neff(w)^{-q},
\end{equation*}
where the final inequality holds since $\lfloor\Neff(w)\rfloor+1 \geq \Neff(w)$. Combining~\eqref{equation:expected-Wasserstein-distance-decomposition-sqrt-sum} with the bounds for $H$ and $R$ completes the proof with $c_0 = c\bigl(1+{1}/\ntspace\ntspace\sqrt{1-2q}\bigr)$.
\end{proof}

Proposition~\ref{proposition:weighted-empirical-Wasserstein-expectation-bound} 
generalizes \parencite[Theorem~1]{fournier2015rate} to the setting of weighted  empirical distributions. The~convergence rates of both results coincide; the only difference is that their bound uses the actual sample size~$N$, whereas ours uses the effective sample size~$\Neff(w)$.

We now turn to the nonstationary case, where the samples~$\rxi_1, \ldots, \rxi_T$ are drawn from differing distributions~$\Prob_1, \ldots, \Prob_T$. As~discussed in Section~\ref{sec:problem-setting}, we assume that the $p$-Wasserstein distances between the historical distributions and the next-period distribution are bounded by a drift profile~$\theta$. We~will make use of the following definition of \emph{normalized weighted cumulative drift}, 
\[
D_p(w,\theta) \defeq \biggl(\ttspace\sum_{t=1}^T w_t\theta_{T-t+1}^p\biggr)^{1/p},
\]
which is the $w$-weighted $L_p$-norm of the drift profile~$\theta$. For~a power-law drift profile of the form~\eqref{eq:power-law-drift-profile}, it~has a minimal value of $\rho$ if all of the weight is placed on the most recent observation and a maximal value of $T^\eta\rho$ if all of the weight is placed on the oldest observation.

\begin{lemma}\label{lemma:shifted-samples-coupled-inequality}
Let $\Prob_1, \ldots, \Prob_{T}, \Prob_{T+1}\in\mathfrak{P}(\Xi)$ satisfy $W_{p}(\Prob_t,\Prob_{T+1})\le\theta_{T-t+1}$ for all $t\in[T]$. Then for weights~$w \in \weights_T$ and independent random samples $\rxi_t \sim \Prob_t$, $\rzeta_t \sim \Prob_{T+1}$, $t\in[T]$,
\[
\Expt\biggl[W_p^p\Bigl(\tMeasureSum_{t=1}^T w_t\pmProb_{\rxi_t},\Prob_{T+1}\Bigr)\biggr]^{1/p}
\leq
\Expt\biggl[W_p^p\Bigl(\tMeasureSum_{t=1}^T w_t\pmProb_{\rzeta_t},\Prob_{T+1}\Bigr)\biggr]^{1/p} + D_p(w,\theta).
\]
\end{lemma}

\begin{proof}
By the triangle inequality for Wasserstein distances (see, e.g., \parencite[page~94]{Optimal-Transport:Villani}),
\begin{align}\label{equation:distribution-shifted-triangle-inequality}
W_p\Bigl(\tMeasureSum_{t=1}^T w_t \pmProb_{\xi_t},\Prob_{T+1}\Bigr) 
   &\leq W_p\Bigl(\tMeasureSum_{t=1}^T w_t \pmProb_{\zeta_t},\Prob_{T+1}\Bigr) + W_p\Bigl(\tMeasureSum_{t=1}^T w_t \pmProb_{\xi_t},\tMeasureSum_{t=1}^T w_t \pmProb_{\zeta_t}\Bigr) \notag\\ 
&\leq W_p\Bigl(\tMeasureSum_{t=1}^T w_t \pmProb_{\zeta_t},\Prob_{T+1}\Bigr) + \Bigl(\tMeasureSum_{t=1}^T w_t \lVert\xi_t - \zeta_t\rVert^p\Bigr)^{1/p}.
\end{align}
We construct a coupling of $(\rxi_1,\ldots,\rxi_T)$ and $(\rzeta_1,\ldots,\rzeta_T)$ under which the final term in~\eqref{equation:distribution-shifted-triangle-inequality} has $L^p$-norm at most~$D_p(w,\theta)$. Since the infimum in the definition of $W_p$ is attained (see, e.g., \parencite[Theorem~4.1]{Optimal-Transport:Villani}), for each $t\in[T]$ there exists a coupling of $\Prob_t$ and $\Prob_{T+1}$ under which
$\Expt\bigl[\|\rxi_t-\rzeta_t\|^p\bigr] = W_p^p(\Prob_t,\Prob_{T+1}) \le \theta_{T-t+1}^p$.
The~product of these couplings yields a coupling of
$\prod_{t=1}^T \Prob_t$ and $\prod_{t=1}^T \Prob_{T+1}$ under which
\begin{equation}\label{equation:expected-drift}
\Expt\biggl[\ttspace\sum_{t=1}^T w_t \|\rxi_t-\rzeta_t\|^p\biggr]
= \sum_{t=1}^T w_t\Expt\Bigl[\|\rxi_t-\rzeta_t\|^p\Bigr]
\le \sum_{t=1}^T w_t \theta_{T-t+1}^p = D_p^p(w,\theta).
\end{equation}
Taking $L^p$-norms on both sides of \eqref{equation:distribution-shifted-triangle-inequality} under this coupling, we obtain
\begin{align*}
\Expt\biggl[W_p^p\Bigl(\tMeasureSum_{t=1}^T w_t\pmProb_{\rxi_t},\Prob_{T+1}\Bigr)\biggr]^{1/p}
&\le \Expt\biggl[\Bigl(W_p\Bigl(\tMeasureSum_{t=1}^T w_t\pmProb_{\rzeta_t},\Prob_{T+1}\Bigr) + \Bigl(\tMeasureSum_{t=1}^T w_t \lVert\rxi_t - \rzeta_t\rVert^p\Bigr)^{\!1/p}\ttspace\Bigr)^{p}\biggr]^{1/p}\\
&\le \Expt\biggl[W_p^p\Bigl(\tMeasureSum_{t=1}^T w_t\pmProb_{\rzeta_t},\Prob_{T+1}\Bigr)\biggr]^{1/p} + \Expt\Bigl[\tMeasureSum_{t=1}^T w_t \lVert\rxi_t - \rzeta_t\rVert^p\Bigr]^{1/p}\\
&\le \Expt\biggl[W_p^p\Bigl(\tMeasureSum_{t=1}^T w_t\pmProb_{\rzeta_t},\Prob_{T+1}\Bigr)\biggr]^{1/p} + D_p(w,\theta),
\end{align*}
where the second inequality follows from Minkowski's inequality and the last uses \eqref{equation:expected-drift}.
\end{proof}

Lemma~\ref{lemma:shifted-samples-coupled-inequality} couples the heterogeneous samples~$\rxi_1, \ldots, \rxi_T$ to an independent and identically distributed proxy~$\rzeta_1, \ldots, \rzeta_T$  and shows that the $L^p$-norms of the distances to the weighted empirical distribution exceed those to the proxy by at most the weighted cumulative drift~$D_p(w,\theta)$. Since the triangle inequality and the coupling bound in the proof are symmetric in the two empirical distributions, the reverse inequality holds as well, and so the two $L^p$-norms differ by at most $D_p(w,\theta)$ in absolute value. The~additive term~$D_p(w,\theta)$ is essentially sharp: one can readily construct simple sequences of distributions for which the inequality in Lemma~\ref{lemma:shifted-samples-coupled-inequality} is tight. 

We now combine the expectation bound of Proposition~\ref{proposition:weighted-empirical-Wasserstein-expectation-bound} with the drift bound of Lemma~\ref{lemma:shifted-samples-coupled-inequality} to study the behavior of weighted empirical distributions that are generated by nonstationary distributions. Under the additional assumption that $\Xi$ is bounded, we can use McDiarmid's inequality to upgrade the resulting expectation bound to give a concentration bound. Note that the boundedness of $\Xi$ implies that every distribution~$\Prob\in\mathfrak{P}(\Xi)$ has bounded moments.

\begin{theorem}[Concentration for distributions with drift]
\label{theorem:time-heterogeneous-concentration-inequality}
Assume $\diam(\Xi)<\infty$. Let~the distributions~$\Prob_1, \ldots, \Prob_{T}, \Prob_{T+1} \in \mathfrak{P} (\Xi)$ satisfy $W_{p}(\Prob_{t},\Prob_{T+1}) \leq \theta_{T-t+1}$ for all $t \in [T]$, and let $q\in(0,1/2)$ achieve \eqref{eq:FG_Thm1} for $\Prob_{T+1}$. Then for independent random samples $\rxi_t\sim\Prob_t$, $t\in[T]$, there exist constants~$c_1,c_2>0$ (depending only on~$m$,~$\diam(\Xi)$,~$p$, and~$q$) such that
\begingroup\setlength{\multlinegap}{0pt}
\begin{multline}\label{equation:time-heterogeneous-concentration-inequality}
    \Pr\biggl[ W_p\Bigl(\tMeasureSum_{t=1}^T w_t \pmProb_{\rxi_t},\Prob_{T+1}\Bigr) \geq \epsilon\biggr]
        \leq \exp\Biggl(\ntspace-c_1\Neff(w)\cdot\biggl(\epsilon^p-\Bigl(D_p(w,\theta) + \bigl(c_2\Neff(w)^{-q}\bigr)^{1/p}\Bigr)^p \biggr)_+^{2}\Biggr)\\
        \quad \text{for all}~T\in\mathbb{N},\; w \in \weights_T, \; \epsilon \in \reals_+.
\end{multline}
\endgroup
In particular, whenever $\epsilon \geq D_p(w,\theta) +  (2c_2\Neff(w)^{-q})^{1/p}$,
\begin{equation}\label{equation:time-heterogeneous-concentration-inequality-clean}
    \Pr\biggl[ W_p\Bigl(\tMeasureSum_{t=1}^T w_t \pmProb_{\rxi_t},\Prob_{T+1}\Bigr) \geq \epsilon\biggr]
        \leq \exp\biggl(-\frac{c_1}{4}\Neff(w)\cdot\Bigl(\epsilon-D_p(w,\theta)\Bigr)^{2p}_+\biggr).
    \tag{\begin{NoHyper}\ref{equation:time-heterogeneous-concentration-inequality}\end{NoHyper}$^{\prime}$}
\end{equation}
\end{theorem}

\begin{proof} 

We prove the result by applying McDiarmid's inequality to the random variable $W_p^p\bigl(\tMeasureSum_{t=1}^{T}w_{t}\pmProb_{\rxi_{t}},\Prob_{T+1}\bigr)$ and bounding its expectation. Lemma~S.1 in the supplementary material, which uses the bounded-differences constants $w_t \diam(\Xi)^p$, $t\in[T]$, gives
\begin{multline*}
\Pr\Biggl[
W_p^p\Bigl(\tMeasureSum_{t=1}^{T}w_{t}\pmProb_{\rxi_{t}},\Prob_{T+1}\Bigr)
- \Expt \biggl[W_p^p\Bigl(\tMeasureSum_{t=1}^{T}w_{t}\pmProb_{\rxi_{t}},\Prob_{T+1}\Bigr)\biggr]
\ge \epsilon \Biggr]
\le \exp \Bigl(- c_1 \Neff(w) \epsilon^{2}\Bigr)\\
\quad \text{for all}~T\in\mathbb{N},\; w \in \weights_T, \; \epsilon \in \reals_+,
\end{multline*}
where $c_1 = 2 \diam(\Xi)^{-2p}$.
It~remains to bound the expectation term. For~independent random samples $\rzeta_t \sim \Prob_{T+1}$, $t\in[T]$, Lemma~\ref{lemma:shifted-samples-coupled-inequality} and the expectation bound of Proposition~\ref{proposition:weighted-empirical-Wasserstein-expectation-bound} yield
\begin{align*}
    \Expt\biggl[W_p^p\Bigl(\tMeasureSum_{t=1}^T w_t \pmProb_{\rxi_t},\Prob_{T+1}\Bigr)\biggr]^{1/p} &\leq \Expt\biggl[W_p^p\Bigl(\tMeasureSum_{t=1}^T w_t \pmProb_{\rzeta_t},\Prob_{T+1}\Bigr)\biggr]^{1/p} + D_p(w,\theta) \\
    &\leq\bigl(c_2\Neff(w)^{-q}\bigr)^{1/p} + D_p(w,\theta).
\end{align*}%
Hence,
\begingroup\allowdisplaybreaks%\mathtoolsset{firstline-afterskip=0pt,lastline-preskip=0pt}
\renewcommand{\minalignsep}{0pt}
\newlength{\wpcol}\settowidth{\wpcol}{$\displaystyle\le{}$}%
\newlength{\wpleft}\setlength{\wpleft}{10pt}
\newlength{\wpright}\setlength{\wpright}{30pt}
\begin{flalign*}
\hspace{\wpleft}& {\Pr\biggl[W_p\Bigl(\tMeasureSum_{t=1}^{T}w_{t}\pmProb_{\rxi_{t}},\Prob_{T+1}\Bigr) \ge \epsilon \biggr]} & \\
\hspace{\wpleft}={}& {\Pr\biggl[W_p^p\Bigl(\tMeasureSum_{t=1}^{T}w_{t}\pmProb_{\rxi_{t}},\Prob_{T+1}\Bigr) \ge \epsilon^p \biggr]} & \\
\hspace{\wpleft}\le{}& \begin{multlined}[t][\dimexpr\displaywidth-\wpcol-\wpleft-\wpright\relax]
{\Pr\Biggl[W_p^p\Bigl(\tMeasureSum_{t=1}^{T}w_{t}\pmProb_{\rxi_{t}},\Prob_{T+1}\Bigr) - \Expt\biggl[W_p^p\Bigl(\tMeasureSum_{t=1}^{T}w_{t}\pmProb_{\rxi_{t}},\Prob_{T+1}\Bigr)\biggr] \Biggr.} \\
\Biggl. \ge \epsilon^p - \Bigl(D_p(w,\theta) + \bigl(c_2\Neff(w)^{-q}\bigr)^{1/p}\Bigr)^p \Biggr]
\end{multlined} & \\
\hspace{\wpleft}\le{}& \begin{multlined}[t][\dimexpr\displaywidth-\wpcol-\wpleft-\wpright\relax]
{\exp\Biggl(\ntspace-c_1\Neff(w)\biggl(\epsilon^p-\Bigl(D_p(w,\theta) + \bigl(c_2\Neff(w)^{-q}\bigr)^{1/p}\Bigr)^p \biggr)_+^{2}\Biggr)} \\
\text{for all}~T\in\mathbb{N},\; w \in \weights_T, \; \epsilon \in \reals_+,
\end{multlined} &
\end{flalign*}
\endgroup
which is \eqref{equation:time-heterogeneous-concentration-inequality}. For~the tail \eqref{equation:time-heterogeneous-concentration-inequality-clean}, note that, since $x\mapsto \bigl(x+D_p(w,\theta)\bigr)^p - x^p$ is nondecreasing on $\reals_+$,
\[
\epsilon^p - \Bigl(D_p(w,\theta) + \bigl(c_2\Neff(w)^{-q}\bigr)^{1/p}\Bigr)^p \ge \Bigl(\epsilon-D_p(w,\theta)\Bigr)^p - c_2\Neff(w)^{-q} \ge \frac{1}{2}\Bigl(\epsilon-D_p(w,\theta)\Bigr)^p_+,
\]
where both inequalities use $\epsilon \ge D_p(w,\theta) + \bigl(2c_2\Neff(w)^{-q}\bigr)^{1/p}$.
\end{proof}

Theorem~\ref{theorem:time-heterogeneous-concentration-inequality} shows that, after accounting for \(D_p(w,\theta)\) and \(\Neff(w)^{-q/p}\), the weighted empirical distribution concentrates about \(\Prob_{T+1}\). The~simplified bound \eqref{equation:time-heterogeneous-concentration-inequality-clean} highlights that once $\epsilon - D_p(w,\theta)$ exceeds a constant multiple of $\Neff(w)^{-q/p}$, the tail probability decays at least as fast as $\exp\bigl(-\frac{c_1}{4}\Neff(w)(\epsilon - D_p(w,\theta))^{2p}\bigr)$.

For $p>m/2$, the bound of Theorem~\ref{theorem:time-heterogeneous-concentration-inequality} matches that of \parencite[Theorem~2]{fournier2015rate}, up to replacing the actual sample size~$N$ by the effective sample size~$\Neff(w)$ and shifting the radius~$\epsilon$ down by the drift term~$D_p(w,\theta)$. Our~bound nevertheless differs from theirs in two respects, both of which stem from the use of McDiarmid's inequality to move from equal to arbitrary weights. First, for $p<m/2$, their result yields tighter tails of order $\exp(-c_1 N\epsilon^{m})$, whereas ours has the form~$\exp(-c_1 N\epsilon^{2p})$. Second, we work on bounded domains to ensure that the bounded differences property holds, whereas \parencite[Theorem~2]{fournier2015rate} treats unbounded domains. The~proof in~\parencite{fournier2015rate} avoids both limitations through a Poissonization of the samples, but this no longer applies with unequal weights. Therefore, a result for weighted samples in full generality remains open, and may be quite technically involved.

\section{Optimal Observation Weightings}\label{sec:optimise-weights}
Theorem~\ref{theorem:time-heterogeneous-concentration-inequality} shows that, for a given ambiguity radius~$\epsilon$, different choices of observation weights~$w$ lead to different probabilistic guarantees that the next-period distribution~$\Prob_{T+1}$ lies within Wasserstein distance~$\epsilon$ of the weighted empirical distribution~\(\tMeasureSum_{t=1}^T w_t \pmProb_{\rxi_t}\). It~is therefore natural to seek a weighting that achieves the strongest probabilistic guarantee at this ambiguity radius, thereby providing the least conservative ambiguity~set.

For simplicity we focus on the tail regime of Theorem~\ref{theorem:time-heterogeneous-concentration-inequality} where \eqref{equation:time-heterogeneous-concentration-inequality-clean} applies. This corresponds to values of~$\epsilon$ that are not too small, and the bound simplifies to
\[
\Pr\biggl[
W_p\Bigl(\tMeasureSum_{t=1}^T w_t \pmProb_{\rxi_t},\Prob_{T+1}\Bigr) \ge \epsilon
\biggr]
\le
\exp\biggl(-\frac{c_1}{4} \Neff(w)\cdot\Bigl(\epsilon - D_p(w,\theta)\Bigr)^{2p}_+\biggr).
\]
In this regime, minimizing the conservatism of the ambiguity set
is equivalent to maximizing the magnitude of the exponent on the right-hand side. For~a fixed radius~$\epsilon$ and drift profile~$\theta$, the factor~$c_1/4$ is independent of~$w$, so its dependence on the domain diameter affects the certified confidence level but not the maximizing weights. Accordingly, we formulate the problem of selecting the optimal weighting~$w$ as
\begin{alignat}{2}\label{problem:measure-concentration-optimization} 
    \maximize_{w\in\weights_T} ~~ \Neff(w)\cdot\Bigl(\epsilon - D_p(w,\theta)\Bigr)_+^{2p}.
\end{alignat}
This objective presents a fundamental \emph{variance--drift} trade-off between two opposing forces when choosing the weights~$w$: on the one hand, \emph{reducing sampling variance} by increasing the effective sample size~$\Neff(w)$; and on the other, \emph{reducing distributional drift} by decreasing the weighted cumulative drift~$D_p(w,\theta)$. Placing more weight on recent samples reduces our exposure to drift but limits the number of effectively independent random samples, and vice versa. Note that for a power-law drift profile of the form~\eqref{eq:power-law-drift-profile}, dividing the objective in~\eqref{problem:measure-concentration-optimization} by~$\rho^{2p}$ or $\epsilon^{2p}$ shows that the optimal weights depend on the relative scales of $\epsilon$ and $\rho$, rather than their individual values.

Problem~\eqref{problem:measure-concentration-optimization} is high-dimensional and nonconvex, and thus appears to be computationally challenging. Fortunately, it admits a  significant simplification: it always possesses an optimal solution that lies within a simple one-parameter family of weights which are truncated affine functions of the $p$-th power of the drift profile.

\begin{theorem}[Structure of the optimal weights]\label{theorem:optimal-weights-p>=1}
Problem~\eqref{problem:measure-concentration-optimization} admits an optimal solution
\[
w_{t} = \Bigl(a_1 - a_2 \theta_{T-t+1}^p\Bigr)_+, \quad t \in [T],
\]
for some constants~\(a_1,a_2\ge 0\), with $a_1=\epsilon D_p(w,\theta)^{p-1} a_2$ whenever $\epsilon>D_p(w,\theta)>0$. 
\end{theorem}

\begin{proof}
For problem~\eqref{problem:measure-concentration-optimization}, we may assume that the positive-part operation in the objective function is inactive at optimality,
otherwise every feasible weighting is optimal, including the uniform weighting which has the stated form with \(a_1=1/T\) and~\(a_2=0\).  Then any optimal solution~\(w\) satisfies \(D_p(w,\theta)<\epsilon\), so it is a local maximizer~of
\[
   \maximize_{w\in\reals^T} ~~ \Neff(w)\Bigl(\epsilon - D_p(w,\theta)\Bigr)^{2p} ~~ \subjectTo ~~ \sum_{t=1}^{T} w_t = 1,~~ w_t \ge 0, ~~t \in [T].
\]
Since the objective function is continuous and the feasible region is compact, optimal solutions exist. Moreover, the linear independence constraint qualification holds at every feasible point: if~\(\cA=\{t\in[T]: w_t=0\}\) is the active set, then $\cA \subsetneq [T]$ since at least one weight must be strictly positive, and thus the gradients of the active inequalities together with the equality-constraint gradient are linearly independent. Consequently, the Karush--Kuhn--Tucker (KKT) conditions are necessary for local optimality. We~claim that the KKT conditions imply the structure~stated~in~the~theorem.

Introducing a multiplier~\(\lambda\) for the equality constraint and multipliers~\(\mu_t \ge 0\) for the inequality constraints, the Lagrangian reads
\[
    L(w,\lambda,\mu)
    = \Neff(w)\Bigl(\epsilon - D_p(w,\theta)\Bigr)^{2p}
      + \lambda\Bigl(\tspace\sum_{t=1}^T w_t - 1\Bigr)
      + \sum_{t=1}^T \mu_t w_t.
\]
By the definitions of $\Neff(w)$ and $D_p(w,\theta)$, and assuming  $D_p(w,\theta)>0$, we have
\[
\frac{\partial \Neff(w)}{\partial w_s} = -2w_s\Neff(w)^2
\quad \text{and} \quad
\frac{\partial D_p(w,\theta)}{\partial w_s} = \frac{\theta_{T-s+1}^p}{p}D_p(w,\theta)^{1-p}.
\]
For strictly positive weights~\(w_s > 0\), complementary slackness implies that \(\mu_s = 0\), and hence the stationarity condition~\(\frac{\partial}{\partial w_s} L(w,\lambda,\mu) = 0\) reduces to
\begin{equation}\label{equation:stationarity} 
-2w_s\Neff(w)^2\Bigl(\epsilon - D_p(w,\theta)\Bigr)^{2p}
-2\Neff(w)\Bigl(\epsilon - D_p(w,\theta)\Bigr)^{2p-1}\theta_{T-s+1}^p D_p(w,\theta)^{1-p}
+\lambda = 0.
\end{equation}
Dividing this by \(2\Neff(w)^2\bigl(\epsilon - D_p(w,\theta)\bigr)^{2p}\) yields $w_s = a_1 - a_2 \theta_{T-s+1}^p$, where
\begin{equation*}
a_1 = \frac{\lambda}{2\Neff(w)^2\bigl(\epsilon - D_p(w,\theta)\bigr)^{2p}}
\quad \text{and} \quad
a_2 = \frac{D_p(w,\theta)^{1-p}}{\Neff(w)\bigl(\epsilon - D_p(w,\theta)\bigr)}.
\end{equation*}
For weights $w_s=0$, the stationarity condition gains $\mu_s\geq 0$ on its left-hand side, and dividing as above gives ${a_1-a_2\theta_{T-s+1}^p\le 0}$.
Thus,  $w_{t} = (a_1 - a_2 \theta_{T-t+1}^p)_+$, $t \in [T]$, holds at optimality.

It~remains to show that 
$a_1=\epsilon D_p(w,\theta)^{p-1}a_2$.
We~multiply the stationarity
condition~\eqref{equation:stationarity} by the weight~$w_s$ and sum over $s\in[T]$. Using the equalities~%
$\sum_{s=1}^T w_s = 1$, $\sum_{s=1}^T w_s^2 = {1}/{\Neff(w)}$, and $\sum_{s=1}^T w_s\theta_{T-s+1}^p = D_p(w,\theta)^p$, this yields
\[
\lambda
= 2\Neff(w)\bigl(\epsilon - D_p(w,\theta)\bigr)^{2p-1}
  \Bigl(\bigl(\epsilon - D_p(w,\theta)\bigr) + D_p(w,\theta)\Bigr)
= 2\epsilon\Neff(w)\bigl(\epsilon - D_p(w,\theta)\bigr)^{2p-1}.
\]
Substituting this expression for $\lambda$ into the definition of $a_1$ gives the result. 

Finally, if ${D_p(w,\theta)=0}$ at optimality, then an optimal $w$ maximizes $\Neff(w)$ among weights supported only on the indices~${\{t\in[T]:\theta_{T-t+1}=0\}}$. This $w$ is uniform on this set, which is of the stated form with $a_2$ sufficiently~large.
\end{proof}

Theorem~\ref{theorem:optimal-weights-p>=1} shows that the optimal weights decay affinely in $\theta^p_{T-t+1}$ and that they are truncated to zero once the drift becomes too large. For~power-law drift profiles of the form~\eqref{eq:power-law-drift-profile}, they therefore decay polynomially in the look-back index and are truncated beyond a finite horizon:
\[
    w_t = \Bigl(a_1-a_2\rho^p(T-t+1)^{\eta p}\Bigr)_+, \qquad t\in[T].
\]
In the special case of $p = 1$ and a linear drift profile, Theorem~\ref{theorem:optimal-weights-p>=1} admits a closed-form solution.

\begin{corollary}\label{cor:optimal-weights-p=1}
Let $p=1$ and consider a linear drift profile of the form~\eqref{eq:linear-drift-profile}. Problem~\eqref{problem:measure-concentration-optimization} admits an optimal solution
\begin{equation}\label{equation:proposition:optimal-weights-p=1}
w_{t}
= \frac{\bigl(\epsilon/\rho-(T-t+1)\bigr)_+}{\sum_{r=1}^{T}\bigl(\epsilon/\rho-(T-r+1)\bigr)_+}, \quad t\in[T], \qquad \text{for}~\epsilon/\rho > 1.
\end{equation}
\end{corollary}

Corollary~\ref{cor:optimal-weights-p=1} shows that for $p=1$ and a linear drift profile, the optimal weights increase linearly with recency.  An~observation $\xi_t$ receives positive weight when $T-t+1<\epsilon/\rho$. Thus, the $\min\{T,\lceil\epsilon/\rho\rceil-1\}$ most recent observations receive positive weight. While the number of positive weights changes only when $\epsilon/\rho$ crosses an integer, their relative values can vary between these thresholds. For~example, when $T \geq 2$, increasing $\epsilon/\rho$ from $5/2$ to $3$ leaves only the two most recent observations with positive weight, but changes their weights from $(1/4,3/4)$ to $(1/3,2/3)$. Larger ambiguity radii tolerate greater distributional drift, making it beneficial to use older  and less relevant observations.

Figure~\ref{figure:weights-for-p=1,2,3,4,5} shows how the optimal weightings behave.
Note that the polynomial order is $\eta p$, but the optimal coefficients~$a_1$ and~$a_2$ in Theorem~\ref{theorem:optimal-weights-p>=1} are more complex and are not fully determined by this product alone. For~$\eta p > 1$, the optimal weights have a concave shape, decreasing at an increasing rate with the age of the observations until they reach zero. In~contrast, when $\eta p< 1$, the optimal weights have a convex shape with a relatively long tail. When $\eta p =1$, the optimal weights lie between these two cases and decrease linearly until they reach zero. The~optimal weights are qualitatively similar if we use the same parameters but a shorter history length~$T$; the weights are truncated and compressed, but their polynomial decay remains the same.

\begin{figure}[tbp]
    \centering
    \includegraphics[width=\dimexpr 42.32pt + 176.16pt + 16pt + 176.16pt + 42.32pt\relax]{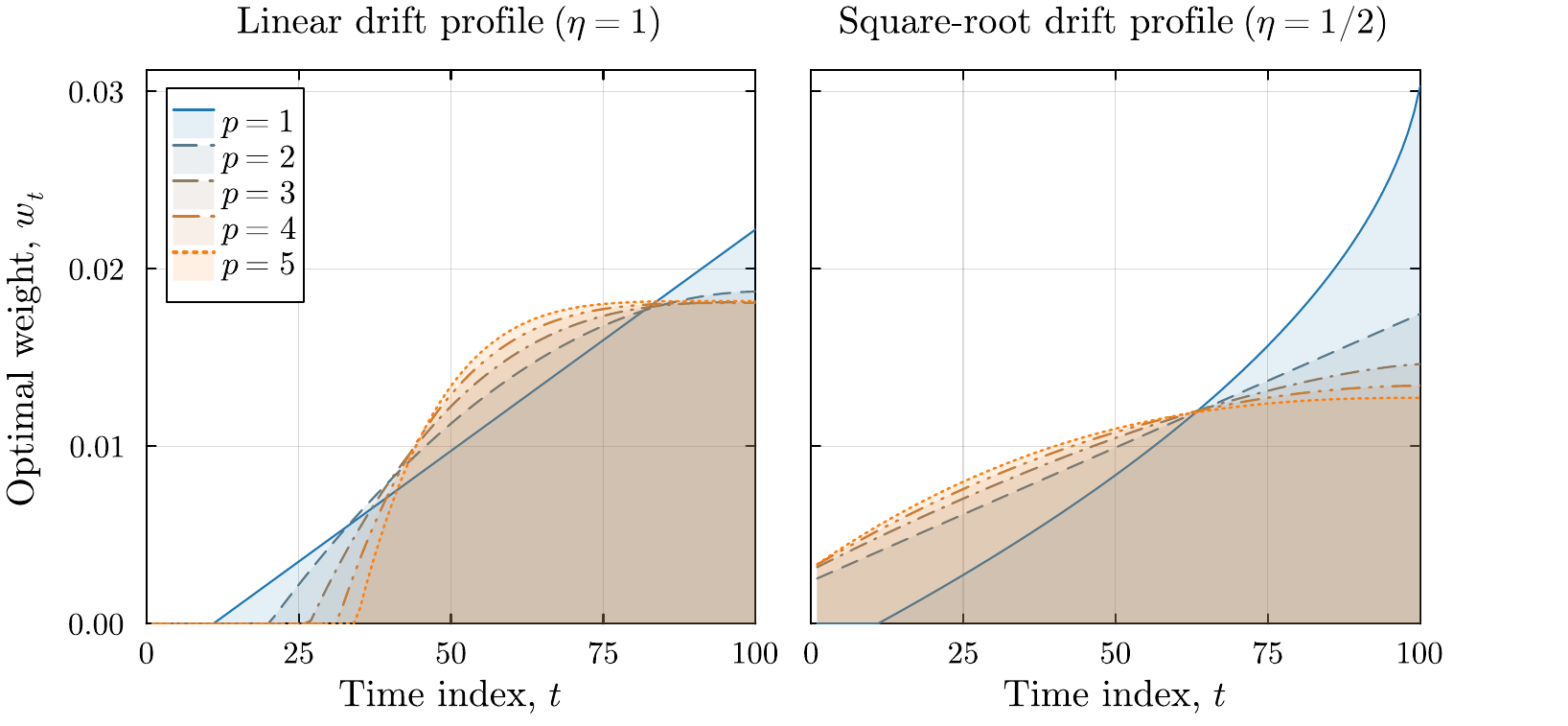}
    \caption{\textbf{Optimal Weightings for Different Wasserstein Orders~$\bm{p}$ and Drift Profiles.} Here we normalize the ratio~$\epsilon/\rho = 90^{\eta}p$ and set $T=100$.}
    \label{figure:weights-for-p=1,2,3,4,5}
\end{figure}

The fact that weight concentrates on more-recent observations as $p$ increases can be explained by the geometry of the Wasserstein distances, as the exponent~$p$ affects how severely large transportation distances are penalized. In~determining the best variance--drift trade-off, larger values of $p$ increase the sensitivity of the $W_p$ distance to the oldest (most drifted) samples. Optimal solutions therefore concentrate weight on a more recent subset of samples to avoid~heavy~penalties.

Theorem~\ref{theorem:optimal-weights-p>=1} also yields an efficient numerical solution method for \eqref{problem:measure-concentration-optimization}. For~each fixed $a_2 \geq 0$, the constraint~$\sum_{t=1}^T w_t = 1$ uniquely determines $a_1 \geq 0$, so the objective of \eqref{problem:measure-concentration-optimization} becomes a one-dimensional function of $a_2$. Thus, the global search for optimal weights reduces to a one-dimensional line search.

We now turn to the question of simple intuitive weighting schemes that are easy to implement and tune. There are two classical schemes for weighting observations. The~first is \emph{windowing}, in which equal weight is placed on the most recent~$s$ observations and older data are discarded. A~window of size~$s\in[T]$ therefore sets
\[
w_{T} = \cdots = w_{T-s+1} = \dfrac{1}{s}, \quad
w_{T-s} = \cdots = w_{1} = 0.
\]
The second is \emph{(simple exponential) smoothing}, in which weights decay geometrically backward in time at a fixed rate. Parametrized by a decay rate~$\alpha\in[0,1]$, this sets
\begin{equation*}\label{equation:smoothing-weights}
w_T=\alpha,\; w_{T-1}=\alpha(1-\alpha),\; w_{T-2}=\alpha(1-\alpha)^2,\;\ldots,
\end{equation*}
where we truncate to the available history length and scale the sum to~1. (For~$\alpha=0$, we take the limiting form as $\alpha\to0$, giving $w_T=\cdots=w_1=1/T$.)

In addition to these classical schemes, we propose a \emph{triangular} weighting scheme
which adopts the form of
Corollary~\ref{cor:optimal-weights-p=1}, so that the weights decrease linearly with the age of the observations until they reach zero. 
More generally,
this is the form of the optimal weights of Theorem~\ref{theorem:optimal-weights-p>=1} for any power-law drift profile with $\eta p = 1$. For~a horizon parameter~$h\in[1,\infty]$,~we~set
\begin{equation}\label{def-triangular-weights}
w_{t} =
\frac{\bigl(h-(T-t+1)\bigr)_+}
     {\sum_{r=1}^{T}\bigl(h-(T-r+1)\bigr)_+},
\quad t\in[T].
\end{equation}
(For $h=1$, we take the limiting form as $h\to 1$, giving $w_T=1,w_{T-1}=\cdots=w_1=0$, and similarly for $h=\infty$, we take $w_T=\cdots=w_1=1/T$.)

Each of these
weighting schemes has a single parameter which can be tuned based on the objective values for problem~\eqref{problem:measure-concentration-optimization}. Our~next two results address this.
Proposition~\ref{proposition:windowing-and-smoothing-near-optimal} gives the optimal window size and smoothing decay rate for $p=1$ and $\eta=1$, and 
Proposition~\ref{prop:triangle-dominance} shows that for $\eta p\in[1,2]$, an optimally tuned triangular scheme performs at least as well as both windowing and smoothing. Proofs
are provided in the supplementary material. Here, for any $x\in\reals$ and interval~$[a,b]\subseteq\reals$, we write $x_{[a,b]}$ to denote the projection of $x$ onto $[a,b]$. 

\begin{proposition}\label{proposition:windowing-and-smoothing-near-optimal}
Let $p=1$ and consider a linear drift profile of the form~\eqref{eq:linear-drift-profile}. Then:
\begin{enumerate}[label={\normalfont{(\roman*)}}, ref={(\roman*)}]
\item \label{proposition:optimal-window-size} An optimal window size for problem~\eqref{problem:measure-concentration-optimization} is either $\lfloor s \rfloor$ or $\lceil s \rceil$, where
\[
s =
\biggl( \frac{1}{3}\Bigl(2\epsilon/\rho-1\Bigr) \biggr)_{[1,T]}.
\]
\item \label{proposition:optimal-decay-rate} An optimal smoothing decay rate for problem~\eqref{problem:measure-concentration-optimization} in the limiting setting of $T=\infty$ is
\[
\alpha = \biggl(\frac{3}{\epsilon/\rho+1}\biggr)_{[0,1]}.
\]
\end{enumerate}
\end{proposition}

There is an interesting connection between the results for the optimal window size and smoothing decay rate.
In~the setting of Proposition~\ref{proposition:windowing-and-smoothing-near-optimal}\ref{proposition:optimal-decay-rate}, windowing with size~$s\in\nats$ and smoothing with decay rate~$\alpha\in[0,1]$ satisfy
\[
\bigl(\Neff,D_1/\rho\bigr) = \Bigl(s,\tfrac{s+1}{2}\Bigr)
\quad\text{and}\quad
\bigl(\Neff,D_1/\rho\bigr) = \Bigl(\tfrac{2}{\alpha}-1,\tfrac{1}{\alpha}\Bigr),
\]
respectively, and both therefore lie on the curve $D_1/\rho = (\Neff + 1)/2$. Since problem~\eqref{problem:measure-concentration-optimization} depends on the weights only through $\Neff$ and $D_1/\rho$, a window of size~$s$ and smoothing at rate~$\alpha = 2/(s+1)$ attain the same objective value. Thus, in terms of the variance--drift trade-off, windowing matches the performance of smoothing restricted to the decay rates~$2/(s+1)$. Consequently, optimally tuned smoothing is always at least as good as optimally tuned windowing, and the two perform equally well when the unconstrained maximizer~$\tfrac{1}{3}(2\epsilon/\rho-1)$ of Proposition~\ref{proposition:windowing-and-smoothing-near-optimal}\ref{proposition:optimal-window-size} is an integer. Otherwise, windowing must round to an adjacent size and is worse than smoothing~for~$\varepsilon/\rho>2$.

\medskip

\begin{proposition}[Triangular weights as a good rule of thumb]
\label{prop:triangle-dominance} Consider a power-law drift profile of the form~\eqref{eq:power-law-drift-profile} with exponent~$\eta\in(0,1]$ and suppose $p\in[1,\infty)$ is such that $\eta p \in [1,2]$. Then, regardless of the values of $\varepsilon, \rho\geq 0$, and $T\in\nats$, an optimally tuned triangular weighting attains:
\begin{enumerate}[label={\normalfont{(\roman*)}}, ref={(\roman*)}]
\item An objective value for problem~\eqref{problem:measure-concentration-optimization} at least as good as that of windowing or smoothing.
\item A minimum fraction ${125}/{128} \approx 98\,\%$
of the optimal value of problem \eqref{problem:measure-concentration-optimization}.
\end{enumerate}
 
\end{proposition}

%\medskip
Figure~\ref{figure:worst-case-weighting-efficiencies} supports  
Proposition~\ref{prop:triangle-dominance} numerically. It~graphs the worst-case efficiency of each tuned simple weighting scheme%
; that is, the smallest ratio of its objective value in problem~\eqref{problem:measure-concentration-optimization} to the optimal value, as a function of the drift exponent~$\eta p$. 
For~$\eta p \ll 1$, the optimal weights are convex and relatively long-tailed (see Figure~\ref{figure:weights-for-p=1,2,3,4,5}), and smoothing, whose weights share this shape, performs best. In~contrast, for $\eta p \gg 1$,
placing weight on old observations incurs a large drift penalty,
and smoothing, which places positive weight on all $T$ observations, performs poorly, whereas windowing benefits from discarding old observations entirely. The~triangular scheme strikes a balance between windowing and smoothing
and performs well in both regimes. Its~worst-case efficiency is at least ${95\,\%}$ for all $\eta p\in[1/4,7/2]$,
beyond the range~$\eta p\in[1,2]$ covered by Proposition~\ref{prop:triangle-dominance}. These results support triangular weightings as a simpler alternative to the parametrized optimal weightings. Triangular weights depend on a single horizon parameter, whereas the optimal weighting family for power-law drift profiles also requires specifying or tuning the exponent~$\eta$.

\begin{figure}[tbp]
    \centering
    \includegraphics[width=\dimexpr 35.04pt + 240pt + 35.04pt\relax]{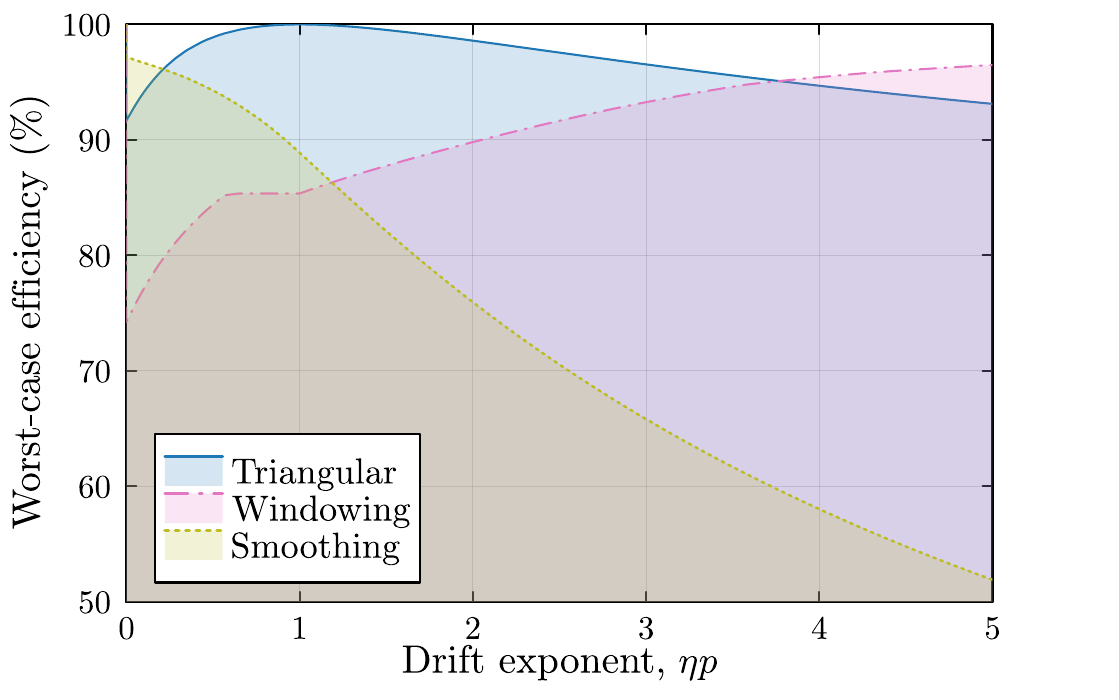}
    \caption{\textbf{Worst-Case Efficiencies of Tuned Simple Weighting Schemes.} The~worst case at each drift exponent value is over $p\in[1,5]$, $\eta\in[0,1]$, $\rho/\varepsilon\in[0,1]$, and $T\in[1000]$.}
    \label{figure:worst-case-weighting-efficiencies}
\end{figure}

The optimal weightings of Theorem~\ref{theorem:optimal-weights-p>=1} yield the smallest (least conservative) ambiguity radius~$\epsilon$ certified by the simplified bound~\eqref{equation:time-heterogeneous-concentration-inequality-clean} to give coverage of the next-period distribution~$\Prob_{T+1}$ at a prescribed confidence level~$1-\beta\in(0,1]$. In~the stationary case $\theta_1=\cdots=\theta_T = 0$, this radius tends to zero as the history length~$T$ increases. In~contrast, when every historical distribution may differ from the next-period distribution, that is, $\theta_1,\ldots,\theta_T > 0$, the ambiguity radius cannot vanish---its size must reflect the drift. 
As~a simple example, set $\Prob_1=\cdots=\Prob_T=\pmProb_{\xi}$ for some~$\xi\in\Xi$ and take $\Prob_{T+1}$ such that $W_p(\pmProb_\xi,\Prob_{T+1}) \eqdef \theta_{\text{min}} > 0$. This sequence of distributions satisfies drift profiles with~$\min_{t\in[T]}\theta_{T-t+1}\geq\theta_{\min}$. Since each observation is equal to $\xi$, any weighted ambiguity set is a ball centered at $\pmProb_\xi$, which fails to contain $\Prob_{T+1}$ for radius~$\epsilon<\theta_{\min}$.
Our~next result quantifies how the minimal confidence radius scales with the drift; for clarity we focus on the case~$p=1$ and a linear drift profile of the form~\eqref{eq:linear-drift-profile}. In~its statement, the constant~$c_1$ from Theorem~\ref{theorem:time-heterogeneous-concentration-inequality} appears in the thresholds for $T$ and $\rho$.

\begin{proposition}[Scaling of the confidence radius]\label{prop:explicit-radius-activation-free}
Under the assumptions of Theorem~\ref{theorem:time-heterogeneous-concentration-inequality}, let $p = 1$ and consider a linear drift profile of the form~\eqref{eq:linear-drift-profile}. For~$\beta \in (0,1)$ and $\rho>0$, define $\epsilon^\star(\beta,\rho)$ as the infimal ambiguity radius~$\epsilon\in\reals_+$ for which
\[
\Pr\biggl[
W_1\Bigl(\tMeasureSum_{t=1}^T w_t \pmProb_{\rxi_t},\Prob_{T+1}\Bigr) \ge \epsilon
\biggr]
\le \beta,
\]
where the weights~$w\in\weights_T$ are optimally chosen.
There exist constants $c^\prime_1,c^\prime_2,c^\prime_3, c^\prime_4 > 0$ (independent of $T$, $\rho$, and $\beta$), such that, if $T \geq  \bigl(\tfrac{12}{c_1}\rho^{-2}\log(1/\beta)\bigr)^{1/3}$, then
\[
\epsilon^\star(\beta,\rho) \le
\begin{cases}
c^\prime_1\Bigl(\rho^{2q}\cdot\log(1/\beta)^{-q}\Bigr)^{1/3} + c^\prime_2\Bigl(\rho\cdot\log(1/\beta)\Bigr)^{1/3} 
& \text{if}~\rho < \rho^{\star}(\beta), \\
\rho + c^\prime_3\cdot\log(1/\beta)^{1/2} + c^\prime_4 
& \text{if}~\rho \ge \rho^{\star}(\beta),
\end{cases}
\]
for a threshold $\rho^{\star}(\beta) = \bigl(\tfrac{12}{c_1}\log(1/\beta)\bigr)^{1/2}$. 
\end{proposition}

\begin{proof}
With $p=1$, $\eta=1$, and $\theta_{T-t+1}=(T-t+1)\rho$, the normalized weighted cumulative drift term is $D_1(w,\theta) = \bigl(\sum_{t=1}^T w_t(T-t+1)\bigr)\rho$; we write this as $D_1(w)\rho$ for shorthand. Theorem~\ref{theorem:time-heterogeneous-concentration-inequality}~yields
\begin{equation}\label{eq:theorem1}
 \Pr\biggl[ W_1\Bigl(\tMeasureSum_{t=1}^T w_t \pmProb_{\rxi_t},\Prob_{T+1}\Bigr) \geq \epsilon\biggr]
        \leq \exp\biggl(-c_1\Neff(w)\Bigl(\epsilon-D_1(w)\rho - c_2\Neff(w)^{-q} \Bigr)_+^{2}\biggr).
\end{equation}
Our goal is to find an upper bound for $\epsilon^\star(\beta,\rho)$ by exhibiting weights~$w$ and a radius~$\epsilon$ for which the right-hand side of~\eqref{eq:theorem1} is at most $\beta$. This occurs when
\begin{equation}\label{eq:theorem2}
\epsilon-D_1(w)\rho - c_2\Neff(w)^{-q}
\geq
\biggl(\frac{\log(1/\beta)}{c_1\Neff(w)}\biggr)^{1/2},
\end{equation}
since in that case the exponent on the right-hand side of~\eqref{eq:theorem1} does not exceed $-\log(1/\beta)$.
 
The optimal weights for problem~\eqref{problem:measure-concentration-optimization} are triangular and given by (\ref{equation:proposition:optimal-weights-p=1}) with a dependence on the value of the ratio~$\epsilon/\rho$. It~will be enough to consider a particular subset of triangular weights,
\[
w_{t}(h) \defeq \frac{\bigl(h-(T-t)\bigr)_+}{\sum_{r=1}^T \bigl(h-(T-r)\bigr)_+}, \quad t\in[T],
\]
supported on the most recent~$h\in[T]$ observations.
A~direct calculation shows that
\[
D_1(w(h))=\frac{h+2}{3}
\quad \text{and} \quad
\Neff(w(h))=\frac{3h(h+1)}{2(2h+1)}  \geq \frac{3h}{4}.
\]
Note that $D_1(w(h))$ and $\Neff(w(h))$ both increase with $h$. To~find an $\epsilon$ for which \eqref{eq:theorem2} holds, we choose $h$ to balance these terms, analyzing the cases $\rho < \rho^{\star}(\beta)$ and $\rho \ge \rho^{\star}(\beta)$ separately.
 
\emph{Case~1: $\rho < \rho^{\star}(\beta)$.}
Set
\[
h^{\star} = \biggl(\frac{12}{c_1}\rho^{-2}\log(1/\beta)\biggr)^{1/3}
\quad\text{and}\quad h=\lceil h^{\star}\rceil.
\]
(The assumption $T \geq \bigl(\tfrac{12}{c_1}\rho^{-2}\log(1/\beta)\bigr)^{1/3}$ ensures that $h = \lceil h^\star \rceil \in [T]$.) We~bound the three terms appearing in~\eqref{eq:theorem2} as follows. First, using that $\rho<\rho^{\star}(\beta)$ is equivalent to $h^\star>1$,
\begin{equation}\label{eq:case1-drift}
D_1(w(h))\rho = \frac{h+2}{3}\rho \le \frac{h^{\star}+3}{3}\rho \le \frac{4}{3}h^\star\rho  = \frac{4}{3} \biggl(\frac{12}{c_1}\rho\log(1/\beta) \biggr)^{1/3}.
\end{equation}
Second, since $\Neff(w(h))\ge 3h/4\ge 3h^\star/4$,
\begin{equation}\label{eq:case1-conc}
\biggl(\frac{\log(1/\beta)}{c_1\Neff(w(h))}\biggr)^{1/2}  \leq \biggl(\frac{4\log(1/\beta)}{3c_1 h^{\star}}\biggr)^{1/2}  = \frac{2}{\sqrt{3}}\biggl(\frac{\rho\log(1/\beta)}{\sqrt{12}\,c_1}\biggr)^{1/3}.
\end{equation}
Third, by the same lower bound on $\Neff(w(h))$,
\begin{equation}\label{eq:case1-bias}
c_2\Neff(w(h))^{-q} \leq c_2\biggl(\frac{4}{3h^\star}\biggr)^q
= c_2\biggl(\frac{4}{3}\biggr)^{q}\biggl(\frac{12}{c_1}\biggr)^{-q/3}\Bigl(\rho^{2q}\log(1/\beta)^{-q}\Bigr)^{1/3}.
\end{equation}
Now, define
\[
c^\prime_1 \defeq c_2\biggl(\frac{4}{3}\biggr)^{q}\biggl(\frac{12}{c_1}\biggr)^{-q/3},
\quad \text{and} \quad
c^\prime_2 \defeq \frac{4}{3}\biggl(\frac{12}{c_1}\biggr)^{1/3} +\frac{2}{\sqrt{3}} \biggl( \frac{1}{\sqrt{12}\,c_1} \biggr)^{1/3}
\]
and set
\[
\epsilon(\beta,\rho)=c^\prime_1\Bigl(\rho^{2q}\log(1/\beta)^{-q}\Bigr)^{1/3}+c^\prime_2\Bigl(\rho\log(1/\beta)\Bigr)^{1/3}.
\]
Adding \eqref{eq:case1-drift}--\eqref{eq:case1-bias} and using the definitions of $c^\prime_1$ and $c^\prime_2$ shows that this choice of $h$ and $\epsilon(\beta,\rho)$ satisfies~\eqref{eq:theorem2}, and therefore the right-hand side of~\eqref{eq:theorem1} is bounded above by~$\beta$.
 
\emph{Case~2: $\rho\ge \rho^{\star}(\beta)$.}
We~set $h=1$ so that $D_1(w(h))\rho=\rho$ and $\Neff(w(h))=1$, and with
\[
\epsilon(\beta,\rho)=\rho+ \biggl(\frac{\log(1/\beta)}{c_1}\biggr)^{1/2} +c_2,
\]
the inequality~\eqref{eq:theorem2} holds as an equality, and hence the right-hand side of~\eqref{eq:theorem1} equals~$\beta$, giving the required result when we take $c^\prime_3=c_1^{-1/2}$, and $c^\prime_4=c_2$.
\end{proof}

For drift that is not too large, i.e., $\rho<\rho^\star(\beta)$, the scaling of the infimal confidence radius~$\epsilon^\star$ with $\rho$ depends on the relative sizes of the terms involving $c^\prime_1$ and $c^\prime_2$ in Proposition~\ref{prop:explicit-radius-activation-free}. The~value of $\rho$ at which these two terms are the same can be calculated as 
\[
\rho_0(\beta)=\bigl(c^\prime_1/c^\prime_2\bigr)^{3/(1-2q)}\cdot \log\bigl(1/\beta\bigr)^{-(1+q)/(1-2q)}.
\]
Since $q\in(0,1/2)$, the $c^\prime_1$ term dominates for $\rho<\rho_0(\beta)$. In~this regime
the radius scales as at most $\rho^{2q/3}$. When $p \ge m/2$, the exponent~$q$ in \eqref{eq:FG_Thm1} can be chosen arbitrarily close to $1/2$ and so scaling is at a rate of approximately $\rho^{1/3}$. In~the intermediate range~$\rho_0(\beta)<\rho<\rho^\star(\beta)$, scaling is always at the rate~$\rho^{1/3}$. Inspecting 
the expressions for  $\rho_0(\beta)$ and $\rho^\star(\beta)$ shows that this intermediate band increases in size as we consider a higher-confidence bound (a smaller $\beta$). 

The proof balances the drift of the positively weighted observations against their sampling error. Specifically, the support length~$h^\star$ is chosen so that the drift~$\rho h^\star$ of these observations matches their sampling error $\bigl(\log(1/\beta)/(c_1 h^\star)\bigr)^{1/2}$, which gives $h^\star\propto\rho^{-2/3}$ and a radius of order~$\rho^{1/3}$ in the intermediate drift range.

It is intriguing that closely related exponents arise in a quite different example of nonstationary stochastic optimization. \textcite{besbes-et-al-2015-nonstationary} consider an application where observations of a cost function~$f_t$ (through its value or gradient) are noisy and the function changes from one period to the next. They show that, when the total variation $\tsum{t=2}{T}\lVert f_t - f_{t-1}\rVert_\infty$ is bounded by~$\mathbb{V}_T$, the minimax regret for convex (resp.\ strongly convex) functions under noisy gradient (resp.\ value) feedback is of order~$\mathbb{V}_T^{1/3}T^{2/3}$, and that it is attained by a policy which restarts every $(T/\mathbb{V}_T)^{2/3}$ periods. A~constant inter-period drift bound~$\rho$ corresponds to a variation bound~$\mathbb{V}_T$ of order~$T\rho$, under which their per-period regret is of order~$\rho^{1/3}$ and their restart length is of order~$\rho^{-2/3}$. For~intermediate drift, with regret playing the role of the ambiguity radius and restart length playing the role of $h^\star$ (the support length of the weights used in the proof), these exponents match those of
Proposition~\ref{prop:explicit-radius-activation-free}. 

For~large drift, i.e., $\rho\geq \rho^\star(\beta)$, the optimal weighting strategy is to place all weight on the most recent observation, the drift dominates, and the infimal confidence radius scales~as~$\rho$. For~comparison, the intersection-of-balls approach,  calibrated via a union bound as in \parencite[Proposition~4]{RychenerEtAl2024HeteroDRO}, prescribes a radius of the form~\((T-t+1)\rho+ c\) for the $t$-{th} ball to achieve~$1-\beta$ confidence. This radius scales linearly in both $T$ and~$\rho$, whereas our weighted scheme achieves a radius that is effectively independent of~$T$, and that scales as~$\rho^{1/3}$ in the intermediate-%
drift range, becoming linear in~$\rho$ only when the drift is large.

\section{Numerical Experiments}\label{sec:num-results}

In this section we study a decision problem faced by a newsvendor supplying random demand that evolves over time. We~approach the problem using DRO and compare our weight-based approach to the intersection-based approach of~\parencite{RychenerEtAl2024HeteroDRO}. We~also compare to plain SAA and smoothing.

Consider a newsvendor supplying $m$ heterogeneous items, with underage cost vector~$\costUnder \in\reals_+^m $ and overage cost vector~$\costOver \in\reals_+^m$, to meet a random demand vector~$\rxi$ supported on $\Xi\subseteq\reals^m$. For~an ambiguity set of demand distributions~$\cP\subseteq\mathfrak{P}(\Xi)$, the DRO order vector solves
\begin{equation}\label{problem:dro-newsvendor}
    \minimize_{x\in\Xi}~~\sup_{\Qrob\in\cP}\Expt_{\Qrob}\Bigl[ \costUnder^\transpose (\rxi-x)_+ + \costOver^\transpose(x-\rxi)_+ \Bigr].
\end{equation}
Here $\cP$ may be a Wasserstein ball centered around a weighted empirical distribution, or an intersection of multiple Wasserstein balls each centered around their own empirical distributions. In~what follows we use the $2$-Wasserstein distance for both types of ambiguity sets and solve the DRO problem~\eqref{problem:dro-newsvendor} using the strong-dual reformulation of~\parencite[Theorem~2]{RychenerEtAl2024HeteroDRO}. 

In our experiments we sample the entries of $\costUnder$ independently and uniformly from $\{3,\ldots,9\}$ and set every entry of $\costOver$ to~$1$. For~the demand process we
model an economy of $1000$ consumers, each of whom independently purchases one unit of one of the $m$ items, or nothing. The~corresponding purchase probabilities are given by one of two time-varying probability vectors~$p_t, q_t \in \weights_{m+1}$, whose $(m+1)$-th entries represent the probability of not purchasing any item. These vectors describe two possible market configurations, and at each time~$t$ the demand follows the mixture distribution
\[
\rxi_t\sim v\cdot\mathop{\text{Multinomial}}(1000,p_t) +  (1-v)\cdot\mathop{\text{Multinomial}}(1000,q_t).
\]
Here $\rxi_t$ collects the first $m$ coordinates of the multinomial draw, while the last coordinate counts the consumers who buy nothing. The~mixture weight~$v$ is held fixed over time and drawn uniformly from $\{0.9,0.95,0.99\}$ in each simulation run (see below). The~more unusual configuration~$q_t$ occurs occasionally. Since each consumer purchases at most one unit, $\rxi_t$ is supported in the bounded domain~$\Xi=[0,1000]^{m}$. The~probability vectors~$p_t$ and $q_t$ evolve independently over time. Their initial values are sampled from a $\text{Dirichlet}(1,\ldots,1)$ distribution, after which the item probabilities follow random walks with independent $\text{Triangular}(-\delta,\delta,0)$ innovations. Initially and after each step, we project the item probabilities onto $\{r\in[0.01,0.99]^m:\sum_{i=1}^m r_i \leq 1\}$ and set the no-purchase probability to $1 - \sum_{i=1}^m r_i$, so that $p_t$ and $q_t$ remain valid probability vectors whose item probabilities are bounded away from $0$ and~$1$.
The~parameter~$\delta \geq 0$ controls 
the degree of nonstationarity. One~can show, by coupling the consumers one by one, that this demand process satisfies an inter-period $2$-Wasserstein-distance drift bound of the form $\rho \leq \sqrt{2000m\delta}+1000\sqrt{m}\,\delta$, which is of order~$\delta$ over the range of~$\delta$ considered below, and thus a linear drift-profile bound of the form~\eqref{eq:linear-drift-profile} applies. For~simplicity we restrict our weight-based approach to this family, although other drift profiles may also fit the data.

To assess performance, we use a training-and-testing approach, averaging over $3000$ simulations. Each simulation uses $T = 100$ total historical demand observations. Newsvendor orders are placed at times $t = 71,\ldots,100$ using the observations~$1,\ldots,t-1$, and each order incurs its cost from the demand~$\rxi_t$. For~each method, the parameter values that minimize this total training cost are then used in a final ordering decision, which incurs a testing cost. This is computed by using an analytic formula for the expected newsvendor loss under a mixture of multinomial demands, and averaging results for the demand probability vectors that apply after $1000$ simulated jumps from time~$t = 100$ to $101$.

Our weight-based approach for the ambiguity set in \eqref{problem:dro-newsvendor} uses two parameters, $\epsilon$~and~$\rho$, with the observation weights in the central empirical distribution chosen according to the characterization of Theorem~\ref{theorem:optimal-weights-p>=1}. For~the intersection-based approach, which uses constituent ambiguity balls centered around each observation, we again use the two parameters~$\epsilon$~and~$\rho$ and set the radius of the \(t\)-th ball to \(\epsilon + (T - t + 1)\rho\). (When the realizations of the observations are such that the intersection is empty, we increase the radii until it is nonempty.) SAA~and smoothing consider only the equally weighted empirical and the decay rate--parametrized weighted empirical on the observations, respectively.

Table~\ref{table:parameter-ranges} shows the parameter ranges searched over when training each approach. Here~for scalars~\(a \leq b\) and a positive integer~\(n\), we use \(\linRange(a,b;n)\) for the set of \(n\) equally spaced values from \(a\) to \(b\), including both endpoints. Similarly, for scalars~$a \leq b \in (0,\infty)$ and a positive integer~\(n\), we use \(\logRange(a,b;n)\) for the set of \(n\) values forming a geometric sequence from \(a\) to \(b\), including both endpoints. This logarithmic spacing is the natural choice for parametrizing the decay~rate~\(\alpha\)~and~the~ratio~\(\rho/\epsilon\).

\begin{table}[tbp]
\centering
\caption{\textbf{Parameter Ranges.}}
\label{table:parameter-ranges}
\begin{tabular}{r c}
\toprule
\midrule
\addlinespace
    Decay rate & $\alpha \in \{0\}\cup\logRange(10^{-4},1;30)$ \\
    Ambiguity radius & $\epsilon \in \sqrt{m}\cdot\Bigl(\{0\}\cup\bigcup_{i=0}^2 \linRange(10^i,10^{i+1};10)\Bigr)$  \\
    Shift bound & $ \rho : \rho/\epsilon\in \{0\}\cup\logRange(10^{-4},1;30)$\footnotemark\\
\addlinespace
\bottomrule
\end{tabular}
\end{table}
\footnotetext{For $\epsilon = 0$ we treat the value of the ratio~$\rho/\epsilon$ as a separate parameter.}

We formulate the dualized version of problem~\eqref{problem:dro-newsvendor} in \texttt{Julia}~\parencite{Julia} and optimize over the Lagrange multipliers. Figure~\ref{figure:method-performance} graphs the average test performance of each approach for a single item and different degrees of nonstationarity. Expected costs are reported as percentage differences relative to smoothing, with negative values indicating lower costs. For~a small degree of nonstationarity, unsurprisingly, SAA performs the best, while the intersection approach performs the worst as it cannot recover the unweighted empirical (see Example~\ref{example:behaviour-of-the-intersection-approach}). Even though smoothing can recover the SAA solution with $\alpha=0$, it performs slightly worse than SAA, since this parameter choice is not always selected during training. In~the same way, our weighted approach, which can also recover the SAA solution with the right choice of parameters, performs slightly worse than~SAA.
As~the degree of nonstationarity increases, SAA performs poorly. The~weighted and intersection approaches both use distributional robustness, which is valuable when the change in the underlying distributions is significant. Indeed, their relative performance against smoothing steadily improves as the degree of nonstationarity increases.  Our~weighted approach performs the best whenever there is at least a moderate amount of drift, and always outperforms the intersection approach.

\begin{figure}[tbp]
    \centering
    \includegraphics[width=\dimexpr 38.06pt + 240pt + 38.06pt\relax]{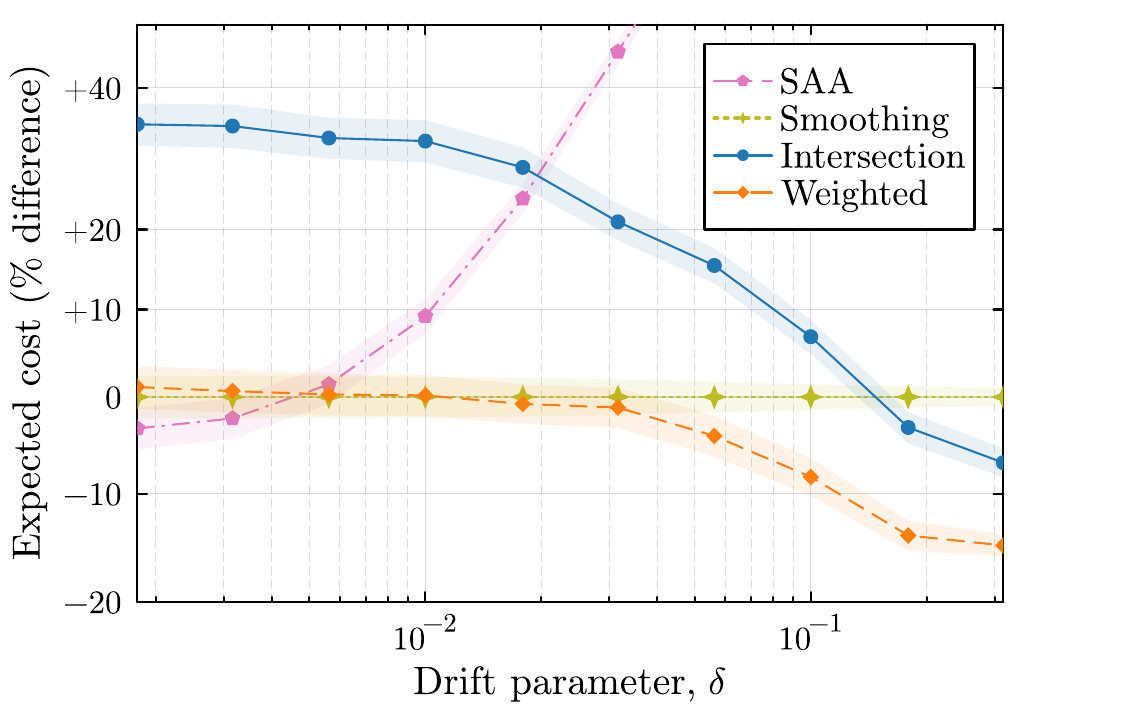}
    \caption{\textbf{Expected Out-of-Sample Performance for a Single Item.} Bands show standard-error ranges. Note that the vertical axis is scaled nonlinearly.}
    \label{figure:method-performance}
\end{figure}

We also evaluated 
our weighted approach with the triangular weights~\eqref{def-triangular-weights} in place of the optimal weights of Theorem~\ref{theorem:optimal-weights-p>=1}. It~performs quite similarly (i.e., to within standard error of ``Weighted'' in Figure~\ref{figure:method-performance}). This supports the observation that a triangular weighting scheme provides a good rule of thumb for use in practice, and is in harmony with the near-optimal efficiency of ``Triangular'' in Figure~\ref{figure:worst-case-weighting-efficiencies}.

Figure~\ref{figure:multi-item-method-performance} graphs the average test performance of each approach for $m\in\{1,\ldots,10\}$ items, with one panel for each degree of nonstationarity $\delta\in\{0.01,0.05,0.1\}$. As~in Figure~\ref{figure:method-performance}, expected costs are reported as percentage differences relative to smoothing.
Moving from a single item to multiple items initially causes a pronounced deterioration in the relative performance of the intersection approach, which is consistent with the high-dimensional effect identified in Example~\ref{example:behaviour-of-the-intersection-approach}. Our~weighted approach performs at least as well as (up to standard error) all of the other approaches throughout, and its advantage over smoothing grows with the number of items, reaching about $11\,\%$ for $\delta=0.05$ and $15\,\%$ for $\delta=0.1$ with 10 items. 
As~in the single-item case, the relative performance of the two approaches using distributional robustness improves as the degree of nonstationarity increases, but in the multi-item experiments it also improves as the number of items grows, apart from the initial deterioration of the intersection approach noted above. The~higher dimensions have the effect of amplifying the effects of nonstationarity. Part of this phenomenon arises from our construction of the nonstationarity, where $\delta$ controls the size of the innovation in each item probability separately. With a greater number of items, a fixed value of $\delta$ results in a larger overall movement of the probability vector from one period to the next (although this is moderated by the projection onto the constrained probability simplex).

\begin{figure}[tbp]
    \centering
    \includegraphics[width=\dimexpr 42.16pt + 112.21pt + 16pt + 112.21pt + 16pt + 112.21pt + 42.16pt\relax]{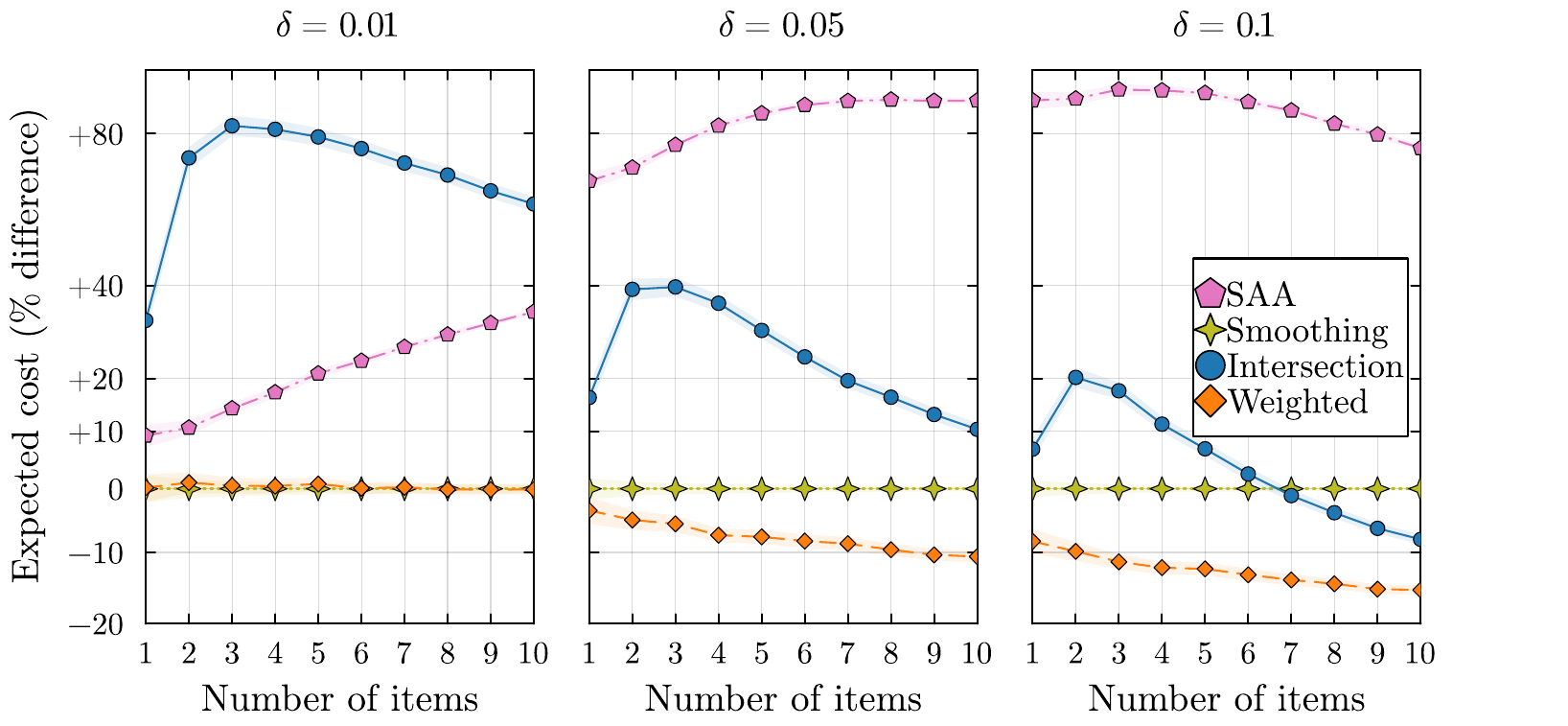}
    \caption{\textbf{Expected Out-of-Sample Performance for Multiple Items and Drifts~$\bm{\delta}$.} Bands show standard-error ranges. Note that the vertical axis is scaled nonlinearly.}
    \label{figure:multi-item-method-performance}
\end{figure}

\section{Extensions}\label{sec:conclusions}
Our paper introduces a framework for DRO with nonstationary data. By~centering Wasserstein balls at weighted empirical distributions and deriving new concentration bounds, we provide a principled approach to trading off statistical variance against~nonstationarity.

We view our contribution as a first step that opens up several promising directions for future research. First, our concentration results bound the probability that the unknown data-generating distribution lies within the ambiguity set, and they therefore inherit the curse of dimensionality. It~would be valuable to investigate whether our analysis can be extended to dimension-free concentration arguments based on the equivalences between DRO and regularization~\parencite{BlanchetKangMurthy2019, gao2023finite}.

Second, our study focuses on Wasserstein distances. There is a rich literature on alternative discrepancy measures between probability distributions, such as $\phi$-divergences, which are particularly well suited for distributions with discrete support~\parencite{BayraksanLove2015, boucheron2003concentration}. The~notion of nonstationarity with bounded drift applies equally to these alternative discrepancy measures, and it is natural to ask what the corresponding analogues of our results would be in these settings.

Finally, our analysis assumes that the historical samples are independent. An~important open question is whether our results can be generalized to settings with controlled dependence, in analogy to classical generalizations of the central limit theorem~\parencite{Billingsley1995, Durrett2019}. We~expect that such extensions are possible, although technically involved.

\section*{Acknowledgements}%%%

The authors thank Andrew Philpott who has contributed through helpful comments and discussions on this work. Part of this work was conducted while the first author was a PhD student at the University of Auckland. The~first author gratefully acknowledges support from the Claude McCarthy Fellowship, the UOCX2117 MBIE Catalyst Fund New Zealand--German Platform for Green Hydrogen Integration~(HINT), and the Imperial College Research Fellowship. The~last author was supported by the Engineering and Physical Sciences Research Council~(EPSRC) grant~EP/W003317/1. High-performance computing resources were provided by the Imperial College Research Computing Service (link:~\url{http://doi.org/10.14469/hpc/2232}).

%\nocite{*}

\printbibliography

\end{document}

\maketitle

\noindent This document proves Lemma~\ref{lem:weighted-bounded-differences} (used in the proof of Theorem~1) and Propositions~\ref{proposition:windowing-and-smoothing-near-optimal} and~\ref{prop:triangle-dominance}. Auxiliary results carry the prefix~``S''.

\section{Lemma~\ref{lem:weighted-bounded-differences}}\label{sec:mcdiarmid}

Recall that $\Neff(w)\defeq\bigl(\ttspace\sum_{t=1}^T w_t^2\bigr)^{-1}$ is the effective sample size of a weighting~$w\in\weights_T$.

\begin{auxlemma}\label{lem:weighted-bounded-differences}
Let~$\diam(\Xi)<\infty$, and $\Prob_1,\ldots,\Prob_T,\Prob_{T+1}\in\mathfrak{P}(\Xi)$. Then for independent random samples~$\rxi_t\sim\Prob_t$, $t\in[T]$,
\begin{multline}\label{eq:weighted-mcdiarmid}
\Pr\Biggl[
W_p^p\Bigl(\tMeasureSum_{t=1}^{T}w_{t}\pmProb_{\rxi_{t}},\Prob_{T+1}\Bigr)
- \Expt \biggl[W_p^p\Bigl(\tMeasureSum_{t=1}^{T}w_{t}\pmProb_{\rxi_{t}},\Prob_{T+1}\Bigr)\biggr]
\ge \epsilon \Biggr]
\le \exp \Bigl(- c_1 \Neff(w) \epsilon^{2}\Bigr)\\
\quad \text{for all}~T\in\mathbb{N},\; w \in \weights_T, \; \epsilon \in \reals_+,
\end{multline}
where $c_1 = 2 \diam(\Xi)^{-2p}$.
\end{auxlemma}

\begin{proof}
We~prove the result by applying McDiarmid's inequality to the random variable $W_p^p\bigl(\tMeasureSum_{t=1}^{T}w_{t}\pmProb_{\rxi_{t}},\Prob_{T+1}\bigr)$. We~first show ${(\xi_1,\ldots,\xi_T) \mapsto W_p^p\bigl(\tMeasureSum_{t=1}^{T}w_{t}\pmProb_{\xi_{t}},\Prob_{T+1}\bigr)}$ satisfies the bounded-differences condition, that is, we bound
\begin{equation}\label{equation:mcD-to-bound}
\biggl|W_p^p\Bigl(\tMeasureSum_{t=1}^{T}w_t\pmProb_{\xi_t},\Prob_{T+1}\Bigr) - W_p^p\Bigl(\tMeasureSum_{t=1}^{T}w_t\pmProb_{\xi^{\prime}_t},\Prob_{T+1}\Bigr)\biggr|,
\end{equation}
for all $(\xi_1,\ldots,\xi_T),(\xi^{\prime}_1,\ldots,\xi^{\prime}_T)\in\Xi^{T}$ with $\xi_t=\xi^{\prime}_t$ for all $t \neq j$ and~some~$j \in [T]$.

Set~$\mu^{\prime} \defeq \tMeasureSum_{t=1}^{T} w_t \pmProb_{\xi^{\prime}_t}$, and let $\gamma^{\prime}$
be an optimal coupling of $\mu^{\prime}$ and $\Prob_{T+1}$, thus achieving cost $W_p^p\bigl(\tMeasureSum_{t=1}^{T}w_t\pmProb_{\xi^{\prime}_t},\Prob_{T+1}\bigr)$.
Since $\mu^{\prime}$ is discrete, there exist conditional laws~$\beta_t^{\prime}\in\mathfrak P(\Xi)$ such that
\[\gamma^{\prime}=\tMeasureSum_{t=1}^{T}w_t\,\pmProb_{\xi^{\prime}_t}\otimes\beta_t^{\prime}.\]
Define a coupling~$\gamma$ of $\mu\defeq\tMeasureSum_{t=1}^T w_t \pmProb_{\xi_t}$ and $\Prob_{T+1}$ by
\[\gamma\defeq\tMeasureSum_{t=1}^{T}w_t\,\pmProb_{\xi_t}\otimes\beta_t^{\prime}.\]
Since $\xi_t=\xi^{\prime}_t$ for all $t\neq j$, the coupling~$\gamma$
differs from~$\gamma^{\prime}$ only in the $j$-th term, where $\xi^{\prime}_j$ is replaced by~$\xi_j$.
The cost of the optimal coupling~$\gamma^{\prime}$ is
\[
W_p^p\Bigl(\tMeasureSum_{t=1}^{T}w_t\pmProb_{\xi^{\prime}_t},\Prob_{T+1}\Bigr)
= \sum_{t=1}^T w_t\Expt_{\beta_t^{\prime}}\bigl[\|\xi^{\prime}_t-\rzeta\|^p\bigr].
\]
For~the feasible coupling~$\gamma$, we analogously obtain
\[
W_p^p\Bigl(\tMeasureSum_{t=1}^{T}w_t\pmProb_{\xi_t},\Prob_{T+1}\Bigr)
\le \sum_{t\neq j} w_t\Expt_{\beta_t^{\prime}}\bigl[\|\xi^{\prime}_t-\rzeta\|^p\bigr]
+ w_j\Expt_{\beta_j^{\prime}}\bigl[\|\xi_j-\rzeta\|^p\bigr].
\]
Subtracting the two expressions yields
\begin{align*}
W_p^p\Bigl(\tMeasureSum_{t=1}^{T}w_t\pmProb_{\xi_t},\Prob_{T+1}\Bigr)
- W_p^p\Bigl(\tMeasureSum_{t=1}^{T}w_t\pmProb_{\xi^{\prime}_t},\Prob_{T+1}\Bigr) &\le w_j\Bigl(\Expt_{\beta_j^{\prime}}\bigl[\|\xi_j-\rzeta\|^p\bigr]
- \Expt_{\beta_j^{\prime}}\bigl[\|\xi^{\prime}_j-\rzeta\|^p\bigr]\Bigr)\\
&\le w_j\diam(\Xi)^p,
\end{align*}
since $0\le\|\xi-\zeta\|^p\le \diam(\Xi)^p$ for all $\xi,\zeta\in\Xi$.
A~symmetric argument then
bounds~\eqref{equation:mcD-to-bound}.
Applying McDiarmid's inequality with the bounded-differences constants~$a_t \defeq w_t \diam(\Xi)^p$, so that $\sum_{t=1}^T a_t^2=\diam(\Xi)^{2p}\sum_{t=1}^T w_t^2=2/(c_1\Neff(w))$, proves~the~result.
\end{proof}

In~the special case~$\Prob_t=\Prob_{T+1}$ and $w_t=1/T$ for all $t\in[T]$, \textcite[Proposition~20]{weed2019sharp} establish the corresponding concentration bound using a duality-based bounded-differences argument. In~contrast, Lemma~\ref{lem:weighted-bounded-differences} allows arbitrary weights and possibly nonidentically distributed samples, and its proof uses a direct coupling~construction.

\section{Proofs of Propositions~\ref{proposition:windowing-and-smoothing-near-optimal} and~\ref{prop:triangle-dominance}}\label{sec:prop-windowing-smoothing}

Recall that the optimal weight-selection problem is
\begin{equation}\label{problem:measure-concentration-optimization}
    \maximize_{w\in\weights_T} ~~ \Neff(w)\cdot\Bigl(\epsilon - D_p(w,\theta)\Bigr)_+^{2p}.
\end{equation}
Throughout, we consider power-law drift profiles
\begin{subequations}\label{eq:drift-profiles}
\begin{equation}\label{eq:power-law-drift-profile}
    \theta_{T-t+1} = (T-t+1)^\eta\rho, \qquad t\in[T],
\end{equation}
for $\eta\in(0,1]$ and $\rho\ge0$, with linear drift profiles
\begin{equation}\label{eq:linear-drift-profile}
    \theta_{T-t+1} = (T-t+1)\rho, \qquad t\in[T],
\end{equation}
the special case of $\eta=1$.
\end{subequations}
We reverse the indexing so that $t\in[T]$ denotes the \emph{age} of an observation, with $w_1$ the weight on the most recent observation and $w_T$ the weight on the oldest. An~observation of age~$t$ then has drift bound~$t^\eta\rho$, so that $D_p(w,\theta)=\rho\bigl(\ttspace\sum_{t=1}^T w_tt^{\eta p}\bigr)^{1/p}$.

We~first dispense with two degenerate cases. If~$\rho=0$, the objective function of problem~\eqref{problem:measure-concentration-optimization} is $\Neff(w)\epsilon^{2p}$, which is maximized by a uniform weighting (attained by windowing at size~$s=T$, smoothing at decay rate~$\alpha=0$, and a triangular weighting with horizon~$h=\infty$). If~$\epsilon\le\rho$, then $D_p(w,\theta)\ge\rho\ge\epsilon$ for every weighting, so the objective function vanishes and any weighting is optimal. We~therefore assume that $\epsilon>\rho>0$; the results we establish hold trivially~otherwise.

\setcounter{proposition}{1}

\begin{proposition}\label{proposition:windowing-and-smoothing-near-optimal}
Let~$p=1$ and consider a linear drift profile of the form~\eqref{eq:linear-drift-profile}. Then:
\begin{enumerate}[label={\normalfont{(\roman*)}}, ref={(\roman*)}]
\item \label{proposition:optimal-window-size} An optimal window size for problem~\eqref{problem:measure-concentration-optimization} is either $\lfloor s \rfloor$ or $\lceil s \rceil$, where
\[
s =
\biggl( \frac{1}{3}\Bigl(2\epsilon/\rho-1\Bigr) \biggr)_{[1,T]}.
\]
\item \label{proposition:optimal-decay-rate} An optimal smoothing decay rate for problem~\eqref{problem:measure-concentration-optimization} in the limiting setting of $T=\infty$ is
\[
\alpha = \biggl(\frac{3}{\epsilon/\rho+1}\biggr)_{[0,1]}.
\]
\end{enumerate}
\end{proposition}

\begin{proof}[Proof of Proposition~\ref{proposition:windowing-and-smoothing-near-optimal}]
\emph{Part~\ref{proposition:optimal-window-size}:} For~windowing with $s\in[T]$, problem~\eqref{problem:measure-concentration-optimization} becomes
\begin{equation*}
    \maximize_{s\in[T]} ~~ F(s) \defeq s\Bigl(\varepsilon-\tfrac{1}{s}\tsum{t=1}{s} t\rho\Bigr)^2_+ = s\Bigl(\varepsilon-\tfrac{1}{2}(s+1)\rho\Bigr)_+^2.
\end{equation*}
Here $F(s) = 0$ for $s \geq 2\varepsilon/\rho-1$, and thus this problem is equivalent to
\begin{equation}\label{problem:optimize-windowing-with-restriction}
    \maximize_{s\in[T]} ~~ G(s) \defeq s\Bigl(\varepsilon-\tfrac{1}{2}(s+1)\rho\Bigr)^2 ~~ \subjectTo ~~  s \leq 2\varepsilon/\rho-1.
\end{equation}
This objective function is initially increasing and has at most one turning point on $(-\infty,2\varepsilon/\rho-1]$. Differentiating, we obtain the unconstrained first-order optimality condition
\begin{equation*}
    \frac{\drv G(s)}{\drv s} = \Bigl(\varepsilon-\tfrac{1}{2}(s+1)\rho\Bigr)^2-s\rho\Bigl(\varepsilon-\tfrac{1}{2}(s+1)\rho\Bigr) = 0,
\end{equation*}
which has the solutions~$\frac{1}{3}(2\varepsilon/\rho-1)$ and $2\varepsilon/\rho-1$. The~former maximizes $G$ on $(-\infty,2\varepsilon/\rho-1]$. Since $G$ increases up to this maximizer and decreases thereafter on this interval, an optimal integer window size is one of the two integers adjacent to its projection~onto~$[1,T]$.

\emph{Part~\ref{proposition:optimal-decay-rate}:} For~smoothing with $\alpha\in(0,1]$, in the limiting setting of $T=\infty$ the objective function of problem~\eqref{problem:measure-concentration-optimization} becomes
\begin{equation*}
\Bigl(\tsum{t=1}{\infty} \bigl(\alpha(1-\alpha)^{t-1}\bigr)^2\Bigr)^{-1} \Bigl(\varepsilon-\tsum{t=1}{\infty}\alpha(1-\alpha)^{t-1}t\rho\Bigr)_+^2 = \Bigl(\tfrac{2}{\alpha}-1\Bigr)\Bigl(\varepsilon-\tfrac{1}{\alpha}\rho\Bigr)_+^2\eqdef F(\alpha).
\end{equation*}
Here $F(\alpha)=0$ for $\alpha\le\rho/\varepsilon$, and thus maximizing $F$ over $(0,1]$ is equivalent to
\begin{equation}\label{problem:optimize-smoothing-with-restriction}
    \maximize_{\alpha\in(0,1]} ~~ G(\alpha) \defeq \Bigl(\tfrac{2}{\alpha}-1\Bigr)\Bigl(\varepsilon-\tfrac{1}{\alpha}\rho\Bigr)^2 ~~ \subjectTo ~~ \alpha\geq\rho/\varepsilon.
\end{equation}
This objective function is initially increasing and has at most one turning point on $[\rho/\varepsilon,\infty)$. Differentiating, we obtain the unconstrained first-order optimality condition
\begin{equation*}
    \frac{\drv G(\alpha)}{\drv \alpha} = \frac{2}{\alpha^2}\biggl(\rho\Bigl(\tfrac{2}{\alpha}-1\Bigr)\Bigl(\varepsilon-\tfrac{1}{\alpha}\rho\Bigr)-\Bigl(\varepsilon-\tfrac{1}{\alpha}\rho\Bigr)^2\biggr) = 0,
\end{equation*}
which has the solutions~$\rho/\varepsilon$ and $3/(\varepsilon/\rho+1)$. The~latter maximizes $G$ on $[\rho/\varepsilon,\infty)$. Since $G$ increases up to this maximizer and decreases thereafter on this interval, we conclude that $\alpha=\bigl(3/(\varepsilon/\rho+1)\bigr)_{[0,1]}$ solves~\eqref{problem:optimize-smoothing-with-restriction}.
\end{proof}

\setcounter{proposition}{2}

\begin{proposition}[Triangular weights as a good rule of thumb]
\label{prop:triangle-dominance}
Consider a power-law drift profile of the form~\eqref{eq:power-law-drift-profile} with exponent~$\eta\in(0,1]$ and suppose $p\in[1,\infty)$ is such that $\eta p \in [1,2]$. Then, regardless of the values of $\varepsilon, \rho\geq 0$, and $T\in\nats$, an optimally tuned triangular weighting~attains:
\begin{enumerate}[label={\normalfont{(\roman*)}}, ref={(\roman*)}]
\item \label{prop:triangle-dominance-1}
An objective value for problem~\eqref{problem:measure-concentration-optimization} at least as good as that of windowing~or~smoothing.
\item \label{prop:triangle-dominance-2} A minimum fraction~${125}/{128}\approx 98\,\%$ of the optimal value of problem~\eqref{problem:measure-concentration-optimization}.
\end{enumerate}
\end{proposition}

A~key quantity is the exponent~$\eta p \geq 0$, and for $z\geq0$, we write
\[
M_z(w) \defeq \sum_{t=1}^{T} w_t t^{z} ,
\]
so that a power-law drift profile gives $D_p(w,\theta) = \rho M_{\eta p}(w)^{1/p}$. Since $\rho>0$, the objective function of problem~\eqref{problem:measure-concentration-optimization} becomes $\rho^{2p}F(w)$, where
\begin{equation}\label{eq:objective-normalized}
F(w) \defeq \Neff(w)\cdot\Bigl(\epsilon/\rho - M_{\eta p}(w)^{1/p}\Bigr)_+^{2p}.
\end{equation}

We~first prove two auxiliary results. The~first shows that triangular weights weakly dominate every other weighting~when~$\eta p =1$.
\begin{auxlemma}\label{lem:matched-triangular}
Let~$w\in\weights_T$. Then there is a horizon parameter~$h\in[1,\infty]$ for which the associated triangular weighting~$w'$ satisfies $\Neff(w')=\Neff(w)$ and $M_1(w')\le M_1(w)$.
\end{auxlemma}

\begin{proof}
We~first choose a horizon whose triangular weighting~$w'$ has the same effective sample size as~$w$. We~then compare the weighted mean ages~$M_1(w')$ and $M_1(w)$ using~Cauchy--Schwarz.

Write $w'(h)$ for the triangular weighting associated with a horizon~$h\in[1,\infty]$. For~${h\in(1,\infty)}$, its entries are
\begin{equation*}
w'_{t}(h) =
\frac{(h-t)_+}
     {\sum_{r=1}^{T}(h-r)_+},
\quad t\in[T],
\end{equation*}
with limiting forms $w'(1)=(1,0,\ldots,0)$ and $w'(\infty)=(1/T,\ldots,1/T)$. The~map ${h\mapsto\Neff(w'(h))}$ is continuous on $[1,\infty)$, equals $1$ at $h=1$, and tends to $T$ as $h\to\infty$. If~$\Neff(w)=T$, we may choose $h=\infty$. Otherwise, since $1\le\Neff(w)<T$, there is a finite $H>1$ with $\Neff(w'(H))>\Neff(w)$, and the intermediate value theorem on $[1,H]$ supplies a matching horizon. In~what follows, fix such a horizon~$h$ and write~$w'=w'(h)$.

If~$h=\infty$, then $\Neff(w)=T$, so $w$ is uniform and $w=w'$. Suppose henceforth that $h<\infty$. Since $w'_t\sum_{r=1}^T(h-r)_+ = (h-t)_+$, we have
\[
t \ge h - w'_t\sum_{r=1}^{T}(h-r)_+
\quad\text{for every $t\in[T]$, with equality whenever $w'_t>0$.}
\]
Multiplying this inequality by $w_t$, summing over $t\in[T]$, and using $\sum_{t=1}^T w_t=1$ gives
\[
M_1(w)\ge h-\sum_{t=1}^{T} w_tw'_t\sum_{r=1}^{T}(h-r)_+
.
\]
Consider multiplying instead by $w'_t$ and summing over $t\in[T]$. Since the inequality is tight whenever $w'_t>0$ and multiplying by $w'_t$ removes each strict inequality where $w'_t=0$, we obtain the equality
\[
M_1(w')= h-\sum_{t=1}^{T} (w'_t)^2\sum_{r=1}^{T}(h-r)_+ .
\]
Equality of the effective sample sizes means $\sum_{t=1}^T w_t^2=\sum_{t=1}^T (w'_t)^2$, so Cauchy--Schwarz yields
\[
\sum_{t=1}^{T} w_tw'_t
\le
\biggl(\ttspace\sum_{t=1}^{T} w_t^2\sum_{t=1}^{T} (w'_t)^2\biggr)^{1/2}
=\sum_{t=1}^{T} (w'_t)^2 ,
\]
and combining these relations gives $M_1(w')\le M_1(w)$.
\end{proof}

The~second auxiliary result lets us extend a comparison of the quantities~$M_z$ at the exponents~$z\in\{1,2\}$ to the interval~$z\in[1,2]$.
\begin{auxlemma}
\label{lem:two-sign-changes}
Let
\[
P(z)=\sum_{t=1}^{T} a_t t^{z},
\]
where the nonzero coefficients~$a_t$, $t\in[T]$, have at most two sign changes. Then $P$ has at most two distinct real zeros unless $P$ is identically~zero.
\end{auxlemma}

\begin{proof}
If~there are no sign changes and the coefficients are not all zero, then all nonzero terms have the same sign for every $z\in\reals$, so $P$~has~no~zeros.

We~next consider the case of one sign change. Multiplying all coefficients by $-1$ if necessary, suppose that the coefficients are negative before the sign change and positive afterwards. Let~$t_-$ be the largest age~$t$ with $a_t<0$ and $t_+$ the smallest age~$t$ with $a_t>0$. Choose a real number~$t'\in(t_-,t_+)$ and set
\[
\overline{P}(z)=t'^{-z}P(z)=\sum_{t=1}^{T} a_{t}\biggl(\frac{t}{t'}\biggr)^{z},
\]
so that the zeros of $\overline{P}$ are the zeros of~$P$. Then
\[
\frac{\drv \overline{P}(z)}{\drv z}=\sum_{t=1}^{T} a_{t}\log\biggl(\frac{t}{t'}\biggr)\biggl(\frac{t}{t'}\biggr)^{z}.
\]
Before the sign change, $a_{t}<0$ and $t<t'$, while afterwards $a_{t}>0$ and $t>t'$. Thus every nonzero term in $\drv \overline{P}(z)/\drv z$ is positive, so $\overline{P}$ is strictly increasing and has at~most~one~zero.

Now~suppose that the coefficients~$a_t$ have two sign changes. After multiplying by $-1$ if necessary, the signs have the form~$+,-,+$. Choose a real number~$t'$ strictly between the largest age~$t$ in the first positive block and the smallest age~$t$ in the following negative block, and again set $\overline{P}(z)=t'^{-z}P(z)$. The~coefficients~$a_{t}\log(t/t')$ in $\drv \overline{P}(z)/\drv z$ have signs $-,-,+$ and hence have only one sign change. By~the previous reasoning, $\drv \overline{P}(z)/\drv z$ can have at most one zero. If~$P$ had at least three distinct zeros, then so would $\overline{P}$. By~Rolle's theorem, $\drv \overline{P}(z)/\drv z$ would then have at least two distinct zeros, contradicting the bound of at most one zero established~above.
\end{proof}

With these results in hand, we are now ready to prove Proposition~\ref{prop:triangle-dominance}.

\begin{proof}[Proof of Proposition~\ref{prop:triangle-dominance}]
\emph{Part~\ref{prop:triangle-dominance-1}:} For~each competitor weighting~$w$, Lemma~\ref{lem:matched-triangular} provides a horizon~$h$ whose triangular weighting~$w'$ satisfies $\Neff(w')=\Neff(w)$ and $M_1(w')\le M_1(w)$. To~show that $w'$ attains an objective value for~\eqref{problem:measure-concentration-optimization} at least as large as that of $w$, it suffices to establish $M_{\eta p}(w')\le M_{\eta p}(w)$ for $\eta p\in[1,2]$. For~windowing, we
prove $M_2(w')\le M_2(w)$, and for smoothing, we prove $M_z(w')\le M_z(w)$ for all sufficiently large~$z$. Then Lemma~\ref{lem:two-sign-changes} is used to extend these comparisons to $z\in[1,2]$ by bounding the number of zeros of
\[
P(z) \defeq M_z(w')-M_z(w) = \sum_{t=1}^{T} a_t t^{z}
\quad \text{and}\quad
a_t \defeq w'_t-w_t.
\]

\emph{Windowing:} Fix~a window size~$s\in[T]$, so that $w_t = 1/s$ for $t\in[s]$, with $w_t=0$ otherwise. Thus, $\Neff(w)=s$. The~cases~$s=1$ and $s=T$ are immediate. Suppose therefore that $1<s<T$. By~Lemma~\ref{lem:matched-triangular}, we may choose a horizon~$h$ whose triangular weighting~$w'$ satisfies $\Neff(w')=s$ and $M_1(w')\le M_1(w)$. This gives~$P(1)\le0$. We~claim next that $P(2)\le0$. The~triangular weighting~$w'$ has positive support~$[m]$, where $m=\min\{\lceil h\rceil-1,T\}$, and on it
\begin{equation}\label{eq:triangular-centred}
w'_t = \frac1m + c\biggl(\frac{m+1}{2}-t\biggr)
\quad \text{for some}~c > 0 .
\end{equation}
Since $w'$ is nonuniform, $s=\Neff(w')<m$. A~direct calculation using $\Neff(w')=s$
gives
\[
M_{2}(w')=\frac{(m+1)(2m+1)}{6} -\frac{m+1}{\sqrt{12}}\sqrt{\frac{(m^{2}-1)(m-s)}{s}},
\quad\text{while}\quad
M_{2}(w)=\frac{(s+1)(2s+1)}{6}.
\]
Hence, $M_{2}(w')\le M_{2}(w)$ is equivalent to
\begin{equation}\label{eq:window-poly}
s(m-s)\bigl(2(m+s)+3\bigr)^2 \le 3(m+1)^2(m^2-1).
\end{equation}
Since $m>s>1$, we have $m\ge3$. Now~$\frac83m^2\le3(m^2-1)$, so it suffices to show
\begin{equation}\label{eq:window-poly-relaxed}
s(m-s)\bigl(2(m+s)+3\bigr)^2 \le\frac83m^2(m+1)^2 .
\end{equation}

We~consider two cases. If~$s\le m/2$, then $2(m+s)+3\le3(m+1)$, so
\[
s(m-s)\bigl(2(m+s)+3\bigr)^2 \le9(m+1)^2s(m-s) \le\frac94m^2(m+1)^2.
\]
Since $9/4<8/3$, this gives the required bound~\eqref{eq:window-poly-relaxed}. If~instead $s>m/2$, then
\[2(m+s)+3 \le 2(m+1)(m+s)/m.\]
Thus
\[
s(m-s)\bigl(2(m+s)+3\bigr)^2 \le \frac{4(m+1)^2}{m^2}s(m-s)(m+s)^2 .
\]
But
\[
\frac23m^4-s(m-s)(m+s)^2 = \frac13(m-s)^3(4m-3s) +\frac23m(2s-m)^3\geq 0,
\]
and consequently,
\[
s(m-s)\bigl(2(m+s)+3\bigr)^2 \le \frac{4(m+1)^2}{m^2}\frac23m^4
= \frac83m^2(m+1)^2.
\]
This proves~\eqref{eq:window-poly} and~hence~$P(2)\le0$.

It~remains to interpolate. For~$t\le s$, the coefficient~$a_t=w'_t-w_t$ is affine and decreasing in~$t$, while, for $s<t\le m$, it is positive, and for $t>m$, it is zero. Hence the nonzero coefficients~$a_t$,~$t\in[T]$, have at most two sign changes. Now~$P(0)=0$, and we have already established $P(1)\leq 0$ and $P(2)\leq 0$. So~if $P(z)>0$ for some $z\in(1,2)$, continuity would give one zero in $[1,z)$ and another in $(z,2]$. Together with the zero at $z=0$, this would give at least three distinct zeros, contradicting Lemma~\ref{lem:two-sign-changes}. Therefore, $M_z(w')\leq M_z(w)$ for $1\leq z\leq2$. With $\Neff(w')=\Neff(w)$ by construction, it follows that the triangular weighting attains a value for~\eqref{eq:objective-normalized} (and thus an objective value for \eqref{problem:measure-concentration-optimization}) at least as large as that of the window~weighting.

\medskip

\emph{Smoothing:} Fix~a decay rate~$\alpha\in[0,1]$ and let $w$ be the associated truncated and normalized weighting, so that $w_t$ is proportional to $(1-\alpha)^{t-1}$, $t\in[T]$. The~cases~$\alpha=1$ and $\alpha=0$ are immediate, as is the case~$T\le2$ (every nonincreasing weighting is then triangular). Suppose therefore that $0<\alpha<1$ and $T\ge3$. By~Lemma~\ref{lem:matched-triangular}, we may choose a horizon~$h$ whose triangular weighting~$w'$ satisfies $\Neff(w')=\Neff(w)$ and $M_1(w')\le M_1(w)$. This gives $P(1)\le0$. We~next establish the sign pattern~$-,+,-$ for the nonzero coefficients~$a_t$.

As~before, $w'$ has positive support~$[m]$, on which $w'_t$ is affine and $w_t$ is strictly convex. Thus $a_t=w'_t-w_t$ is concave on $[m]$, so any positive terms form a single block. The~remaining coefficients~$a_t=-w_t$, $t>m$, are negative. Hence the nonzero coefficients can change sign only when entering or leaving the positive block, so there are at most two sign changes. The~coefficients are not all zero, since $a_T<0$ if $m<T$, and strict convexity prevents $w=w'$ if~$m=T$. Since $\sum_t a_t=0$, both signs occur, so there is at least one~sign~change.

To~rule out exactly one sign change, observe that, since the two weightings have the same effective sample size,
\begin{equation}\label{eq:intermediate-sign-change}
\sum_{t=1}^T\bigl(w_t+w'_t-\lambda\bigr)a_t =\sum_{t=1}^T\bigl(w_t+w'_t\bigr)a_t =\sum_{t=1}^T\bigl((w'_t)^2-w_t^2\bigr)=0
\end{equation}
for any constant~$\lambda$. If~the nonzero coefficients~$a_t$ had exactly one sign change, then, as $w_t+w'_t$ is strictly decreasing in~$t$, we could choose $\lambda$ so that $w_t+w'_t-\lambda$ changes sign at the same point. Every nonzero term on the left-hand side of~\eqref{eq:intermediate-sign-change} would then have the same sign, contradicting~\eqref{eq:intermediate-sign-change}. There are therefore exactly two sign changes, and since the positive terms of $a_t$ form a single block, they follow the pattern~$-,+,-$. In~particular, the last nonzero coefficient is negative, and so $P(z)<0$ for all sufficiently~large~$z$.

It~remains to interpolate. We~have $P(0)=0$ and $P(1)\le0$. If~$P(z)>0$ for some $z>1$, continuity and eventual negativity would give a zero in $[1,z)$ and another in $(z,\infty)$. Together with the zero at $z=0$, this would give at least three distinct zeros, contradicting Lemma~\ref{lem:two-sign-changes}. Hence $M_z(w')\le M_z(w)$ for every $z\ge1$, including $z=\eta p$. Since the effective sample sizes agree, the triangular weighting attains an objective value for~\eqref{problem:measure-concentration-optimization} at least as large as that of the~smoothing~weighting.

\medskip

\emph{Part~\ref{prop:triangle-dominance-2}:} We~compare optimal weights $w_t\propto(h^{\eta p}-t^{\eta p})_+$ with triangular weights~$w'_t\propto(h-t)_+$, using the same horizon~$h$. We~first derive a lower bound on $F(w')/F(w)$. We~then show that, for each fixed $h$, this bound is smallest at $\eta p=2$ and is at least~$125/128$~there.

If~the uniform weighting is optimal, then the triangular family attains the optimum at $h=\infty$ and there is nothing to prove. Otherwise, by Theorem~2 of the main paper, an optimal weighting~$w$ is proportional to the vector~$\bar w\in\mathbb{R}^T$ with entries
\[
\bar w_t=\bigl(h^{\eta p}-t^{\eta p}\bigr)_+,
\quad \text{where} \quad h^{\eta p}
=
\frac{\epsilon}{\rho}
M_{\eta p}(w)^{1-1/p},
\]
for some finite horizon~$h>1$. Using this value of $h$, let $w'$ be the associated triangular weighting, which is proportional to the vector~$\bar w'\in\mathbb{R}^T$ with entries~$\bar w'_t=(h-t)_+$.
Both weightings have positive support~$[m]$, where
$m=\min\{\lceil h\rceil-1,T\}$.

In what follows, for vectors~$x, y \in\reals^T$, we write $\langle x,y\rangle=\sum_{t=1}^T x_ty_t$.

\emph{Step 1. Compare the objective values of $w'$ and $w$.} Write $M=M_{\eta p}(w)$ and $M'=M_{\eta p}(w')$.
Since both weightings are supported on $[m]$, we have
\[
h^{\eta p}-M
=
\frac{\langle \bar w,\bar w\rangle}
     {\sum_t\bar w_t}
\quad\text{and}\quad
h^{\eta p}-M'
=
\frac{\langle \bar w',\bar w\rangle}
     {\sum_t\bar w'_t}.
\]
Also, since $\epsilon/\rho = h^{\eta p}M^{1/p-1}$, we have
\[
\frac{\epsilon}{\rho}-M^{1/p}
=
M^{1/p-1}(h^{\eta p}-M)>0.
\]
For~$p>1$, H\"older's inequality gives
\[
M^{1-1/p}(M')^{1/p}
+
\bigl(h^{\eta p}-M\bigr)^{1-1/p}
\bigl(h^{\eta p}-M'\bigr)^{1/p}
\leq h^{\eta p},
\]
and therefore
\[
\begin{aligned}
\frac{\epsilon}{\rho}-(M')^{1/p}
&{}=M^{1/p-1}\Bigl(h^{\eta p}-M^{1-1/p}(M')^{1/p}\Bigr)\\
&\geq
M^{1/p-1}
\bigl(h^{\eta p}-M\bigr)^{1-1/p}
\bigl(h^{\eta p}-M'\bigr)^{1/p}.
\end{aligned}
\]
For~$p=1$, the same inequality follows directly. Dividing this inequality by $\epsilon/\rho-M^{1/p}>0$, raising both sides to the power~$2p$, and using
\[
N_{\mathrm{eff}}(w)
=
\frac{\left(\sum_t\bar w_t\right)^2}
     {\langle\bar w,\bar w\rangle},
\]
which follows from $w_t=\bar{w}_t/\sum_t\bar w_t$, together with the analogous identity for $w'$, we obtain
\begin{equation}\label{ineq:F-ratio}
\frac{F(w')}{F(w)}
\geq
\frac{\langle\bar w',\bar w\rangle^2}
{\langle\bar w',\bar w'\rangle
 \langle\bar w,\bar w\rangle}.
\end{equation}

\emph{Step 2. Reduction to $z=2$.} For~$z\geq1$, let $\bar w(z)\in\mathbb{R}^T$ have entries~$\bar w_t(z)=(h^z-t^z)_+$, so that $\bar w(\eta p)=\bar w$, and let
\[
G(z)
\defeq
\frac{\bigl\langle\bar w',\bar w(z)\bigr\rangle^2}
{\langle\bar w',\bar w'\rangle
 \bigl\langle\bar w(z),\bar w(z)\bigr\rangle}.
\]
Then \eqref{ineq:F-ratio} reads $F(w')/F(w)\geq G(\eta p)$. We~show that $G$ is nonincreasing on $[1,\infty)$, so that $G(\eta p)\geq G(2)$ for $1\leq\eta p\leq2$. Differentiating gives
\begin{equation}\label{eq:G-derivative}
\frac12\frac{\drv}{\drv z}\log\bigl(G(z)\bigr)
=\frac{\bigl\langle\bar w',\tfrac{\drv}{\drv z}\bar w(z)\bigr\rangle}{\bigl\langle\bar w',\bar w(z)\bigr\rangle}
-\frac{\bigl\langle\bar w(z),\tfrac{\drv}{\drv z}\bar w(z)\bigr\rangle}{\bigl\langle\bar w(z),\bar w(z)\bigr\rangle}.
\end{equation}
Observe that, for $t\in[m]$,
\[
\bar w_t(z)=\bar w'_t\frac{h^z-t^z}{h-t}
\quad\text{and}\quad
\frac{\drv}{\drv z}\bar w_t(z)=\bar w_t(z)g_t(z),
\]
where 
\[
g_t(z)\defeq \frac{h^z\log(h)-t^z\log(t)}{h^z-t^z},
\]
and, for~$t>m$, that $\bar w_t(z)=0$.  Here $g_t(z)$ is $1/z$ times the slope of the chord of the convex function ${x\mapsto x\log(x)}$ between $t^z$ and $h^z$, so it is nondecreasing in~$t$. The~right-hand side of~\eqref{eq:G-derivative} is the difference of two weighted averages of $g_t(z)$, with weights~$\bar w'_t\bar w_t(z)$ and ${\bar w_t(z)^2=\bar w'_t\bar w_t(z)(h^z-t^z)/(h-t)}$. By~Chebyshev's sum inequality, the weighted average of a nondecreasing sequence does not decrease when the weights are multiplied by a nondecreasing sequence.  The~expression $(h^z-t^z)/(h-t)$ is the slope of the chord of the convex function~${x\mapsto x^z}$ between $t$ and $h$ and so is nondecreasing in $t$. This shows that the right-hand side of~\eqref{eq:G-derivative} is nonpositive, and $G$ is~nonincreasing.

\emph{Step 3. Bounding $G(2)$.} We~now show that $G(2)\geq125/128$, that is,
\begin{equation}\label{eq:G2-target}
128\bigl\langle\bar w',\bar w(2)\bigr\rangle^2\geq125\langle\bar w',\bar w'\rangle\bigl\langle\bar w(2),\bar w(2)\bigr\rangle.
\end{equation}
Since $\bar w_t(2)=(h^2-t^2)_+=(h-t)_+(h+t)$, we have
\[
\bigl\langle\bar w',\bar w(2)\bigr\rangle=\sum_{t=1}^m(h-t)^2(h+t), \quad \bigl\langle\bar w(2),\bar w(2)\bigr\rangle=\sum_{t=1}^m(h-t)^2(h+t)^2, \quad\text{and}\quad \langle\bar w',\bar w'\rangle= \sum_{t=1}^m(h-t)^2.
\]
By~the standard Faulhaber's formulae for $\sum_{t=1}^m t^k$, $k\leq4$, each of these is polynomial in $h$~and~$m$, and so is the difference of the two sides of~\eqref{eq:G2-target}. Written in terms of $m-1$ and $h-m$, this difference is
\[
128\bigl\langle\bar w',\bar w(2)\bigr\rangle^2-125 \langle\bar w',\bar w'\rangle\bigl\langle\bar w(2),\bar w(2)\bigr\rangle
=
\frac{m^2}{36}Q\Bigl(m-1,h-m\Bigr),
\]
where
\[
\begin{aligned}
Q(x,y) \defeq{}& 30x^5 +\Bigl(210y+97\Bigr)x^4 +\Bigl(240y^3+270y^2+642y+79\Bigr)x^3 \\
&+\Bigl(660y^4+372y^3+1092y^2 +282y+37\Bigr)x^2 \\
&+\Bigl(540y^5+726y^4+864y^3 +294y^2\Bigr)x +108y^6+432y^5+432y^4 .
\end{aligned}
\]
Every coefficient of $Q$ is nonnegative, and $m-1\geq0$ and $h-m>0$, so the right-hand side is nonnegative. This proves~\eqref{eq:G2-target}, and consequently
\[
\frac{F(w')}{F(w)}
\geq
G(\eta p)
\geq
G(2)
\geq
\frac{125}{128}.
\]
Since $w'$ is a particular triangular weighting, an optimally tuned triangular weighting does at least as well, which proves~the~result.
\end{proof}

\printbibliography